\documentclass[a4paper]{amsart}
\usepackage[utf8]{inputenc}
\usepackage[T1]{fontenc}
\usepackage{lmodern}
\usepackage{amsmath, amsthm, amssymb, amsfonts}
\usepackage[all]{xy}
\usepackage{nicefrac,mathtools}
\usepackage[shortlabels, inline]{enumitem}
\usepackage{microtype}
\usepackage{hyperref}
\usepackage{graphicx}
\usepackage[all]{xy}
\usepackage{xcolor}
\usepackage{tikz-cd}
\usepackage[lite]{amsrefs}
\usepackage{thmtools}

\numberwithin{equation}{section}
\theoremstyle{plain}
\newtheorem{theorem}[equation]{Theorem}

\newtheorem{lemma}[equation]{Lemma}

\newtheorem{proposition}[equation]{Proposition}

\newtheorem{corollary}[equation]{Corollary}

\theoremstyle{definition}
\newtheorem{definition}[equation]{Definition}

\theoremstyle{remark} \newtheorem{remark}[equation]{Remark}
 \newtheorem*{remark*}{Remark}
\newtheorem*{remarks*}{Remark} \newtheorem{example}[equation]{Example}

\newcommand*{\Contc}{\mathrm{C_c}}
\newcommand*{\Cc}{\mathrm{C_c}}
\newcommand{\Contz}{\mathrm{C_0}} \newcommand{\base}[1][G]{{#1}^{(0)}}
\newcommand{\dd}{\mathrm{d}} 
 \newcommand{\supp}{{\textup{supp}}}

\newcommand*{\Hilm}[1][H]{\mathcal #1}
\newcommand{\Bound}{\mathbb B}
\newcommand*{\Cst}{\mathrm{C}^*}
\newcommand{\A}{\mathcal{A}}
\newcommand{\KMS}[1][\beta]{\textup{KMS}\textsubscript{\(#1\)}}
\newcommand{\Ind}{\mathrm{Ind}}

\newcommand*{\inpro}[2]{\langle#1, #2\rangle}
\newcommand*{\Linpro}[2]{\langle\!\langle#1, #2\rangle\!\rangle}

\newcommand{\norm}[1]{\lvert\!\lvert #1\rvert\!\rvert}
\newcommand{\C}{\mathbb{C}}
\newcommand{\R}{\mathbb{R}}
\newcommand{\Z}{\mathbb{Z}}
\newcommand{\N}{\mathbb{N}}
\newcommand{\e}{\mathrm{e}}
\newcommand{\Mat}{\mathrm{M}}

\newcommand*{\nb}{\nobreakdash}
\newcommand*{\defeq}{\mathrel{\vcentcolon=}}
\newcommand{\etale}{{\'e}tale}

\title[Ground states on Fell bundle \(\mathrm{C}^*\)-algebras]{Ground states on the \(\mathrm{C}^*\)-algebras of Fell bundles over \'etale groupoids}

\author{Md Amir Hossain}

\address{Theoretical Statistics and Mathematics Unit, Indian Statistical Institute, Delhi Centre,  7, S. J. S. Sansanwal Marg, New Delhi, 110016, India}
\email{mdamirhossain18@gmail.com}


\begin{document}
\begin{abstract}
Let \(p\colon \mathcal{A}\to G\) be a Fell bundle over a locally compact Hausdorff second countable \'etale groupoid \(G\) and let \(\mathrm{C}^*(G; \mathcal{A})\) be the associated Fell bundle \(\mathrm{C}^*\)-algebra. Suppose \(\sigma^c\) is the real dynamics on \(\mathrm{C}^*(G;\mathcal{A})\) induced by a real-valued \(1\)-cocycle \(c\). We establish an affine homeomorphism between the \(\sigma^c\)-ground state space of \(\mathrm{C}^*(G;\mathcal{A})\) and the state space of \(\mathrm{C}^*(G(Z); \mathcal{A}|_{G(Z)})\), where \(Z\) is the boundary set of \(c\) and \(G(Z)\) is the boundary groupoid. In particular, \(\mathrm{C}^*(G; \mathcal{A})\) admits \(\sigma^c\)-ground states if and only if \(Z\neq \emptyset\). We further investigate the relation between ground states and KMS\(_{\infty}\) states and give a crystallization interpretation of the boundary \(\mathrm{C}^*\)-algebra \(\mathrm{C}^*(G(Z); \mathcal{A}|_{G(Z)})\) when the cocycle is locally constant. Finally, we apply our results to several classes of examples, including twisted \(\mathrm{C}^*\)-algebras of Deaconu--Renault groupoids and twisted higher-rank graph \(\mathrm{C}^*\)-algebras, obtaining explicit descriptions of their ground states.  
\end{abstract}

\maketitle

\section{Introduction}

Ground states of \(\Cst\)\nb-dynamical systems have been extensively studied in connection with the zero-temperature behaviour of the equilibrium states. 
Although every weak\(^*\)-limit of
KMS\(_\beta\) states as \(\beta\) goes to infinity is a ground state, in general  the class of ground states is strictly larger. Therefore, a natural question is to determine the ground state space directly from the underlying \(\Cst\)\nb-dynamical system. This question has been investigated for several classes of \(\Cst\)\nb-algebras. For example, Thomsen~\cite{Thomsen-Ground-states-for-gen-gauge-action-UHF}
studied ground states of a UHF algebra for generalized gauge actions and
identified ground states with the state space of a unital AF-algebra.
A particularly relevant result for the present work is due to
Laca--Larsen--Neshveyev~\cite{Laca-Larsen-Neshveyev-Ground-states}.
They characterized the ground state space of the \(\Cst\)\nb-algebra of an \'etale groupoid~\(G\), for the real dynamics induced by a continuous
real-valued \(1\)\nb-cocycle \(c\), in terms of the state space of the
\(\Cst\)\nb-algebra of the associated boundary groupoid \(G(Z)\). Their
result shows that the geometry of the boundary set \(Z\) determined by the cocycle \(c\) captures the entire ground state space. 
They applied this result to study ground states on certain arithmetic subalgebras of Hecke \(\Cst\)\nb-algebras.

Motivated by this result (\cite{Laca-Larsen-Neshveyev-Ground-states}), we extend the boundary groupoid description of
ground states to \(\Cst\)\nb-algebras arising from a Fell bundle \(p\colon \A\to G\) over
an \'etale groupoid \(G\).
 Our main result (Theorem~\ref{thm-main-ground-st}) shows that the ground state space is
determined by the restricted Fell bundle over the boundary groupoid \(G(Z)\defeq r^{-1}(Z)\cap s^{-1}(Z).\)
More precisely, assuming that \(G(Z)\) is \'etale, we prove that there is an
affine homeomorphism between the state space of
\(\Cst(G(Z);\A|_{G(Z)})\) and the \(\sigma^c\)\nb-ground state space of
\(\Cst(G;\A)\). Under this correspondence, a state \(\varphi\) on
\(\Cst(G(Z);\A|_{G(Z)})\) gives the unique \(\sigma^c\)-ground state
\(\psi_\varphi\) satisfying
\[
\psi_\varphi(f)=\varphi(f|_{G(Z)}),
\]
for \(f\in\Cc(G;\A).\)
In particular, if \(Z=\emptyset\), then \(\Cst(G;\A)\) admits no
\(\sigma^c\)\nb-ground states. Thus, the boundary groupoid provides a
canonical ``zero-energy'' part of the Fell bundle which completely
determines the ground state space.

Our proof is based on two main ideas. For the construction of ground
states from states on \(\Cst(G(Z);\A|_{G(Z)})\), we use the imprimitivity
bimodule associated with an equivalence of Fell bundles between
\(p\colon\A|_{G(Z)}\to G(Z)\) and a Fell bundle over the imprimitivity
groupoid \(G(Z)^G\) (see Section~\ref{sec-induc-state}). This gives an
induction procedure which allows us to construct states on \(\Cst(G;\A)\) form states on \(\Cst(G(Z); \A|_{G(Z)})\)
(see Theorem~\ref{thm-ind-state}). 
For the
converse direction, we use the Muhly--Williams disintegration theorem for
Fell bundles~\cite{Muhly-Williams2008Equivalence-and-disinte-thm-Fell-bundle}. The ground state condition forces the cyclic vector in the
disintegrated representation to be supported on the boundary set \(Z\).
Consequently, the representation restricts to the Fell bundle over
\(G(Z)\) produces the state which determines the original ground state.

When the \(1\)-cocycle \(c\) is locally constant, the above description of ground states admits another interpretation in terms of crystallization. For an almost periodic \(\Cst\)\nb-dynamical system, Laca--Neshveyev--Yamashita~\cite{Laca-Neshveyev-Yamashita-Crystallization} introduced the crystal,
which is a quotient of the fixed-point algebra and whose state space
parametrizes the ground states. We show that, for the dynamics on a Fell
bundle \(\Cst\)\nb-algebra induced by a locally constant \(1\)-cocycle, the crystal is canonically isomorphic to \(\Cst(G(Z);\A|_{G(Z)}).\)
Thus, in this setting, the restricted Fell bundle over the boundary groupoid appearing in our main theorem has an intrinsic \(\Cst\)\nb-algebraic interpretation in terms of crystal of the dynamical system.

A smaller and well-behaved class of ground states, called KMS\(_\infty\) states, was introduced by Connes--Marcolli (\cite{Connes-Marcolli-2006-QSM_Q-lattice}) in their study
of quantum statistical mechanics. These states may be viewed as
weak\(^*\)-limits of KMS\(_\beta\) states as the temperature goes to zero, and hence represent ground states arising as zero-temperature limits of equilibrium states. KMS\(_\infty\) states subsequently have been studied in several classes of \(\Cst\)\nb-algebras, including graph
\(\Cst\)\nb-algebras~\cite{Thomsen-KMS_weight-Graph-alg-Ground-state}, \(\Cst\)\nb-algebras associated to Cayley graphs~\cite{Christensen-Klaus2018Cayley-graphs}, and simple unital AF-algebras~\cite{Thomsen-Equi-states-temp-goes-to-zero}. 
Although every KMS\(_\infty\) state is a ground state, the converse need not be true. Thus, once the ground states have been characterized, it is
natural to ask which of them arise as zero-temperature limits of KMS states.

In the present setting, our description of ground states allows us to give a partial answer to this question. Assume that \(c^{-1}(0)\) is an
\'etale groupoid. We show that if \(\varphi\) is a state on
\(\Cst(G(Z);\A|_{G(Z)})\) and the corresponding ground state
\(\psi_{\varphi}\) is a KMS\(_\infty\) state, then \(\varphi\) must be
tracial. Consequently, whenever \(\Cst(G(Z);\A|_{G(Z)})\) admits a non-tracial state, the corresponding
ground state is not a KMS\(_\infty\) state. 

We conclude with several applications illustrating the main result. First, we consider twisted \(\Cst\)\nb-algebras of Deaconu--Renault groupoids. For the
canonical gauge cocycle, we show that there are no
\(\sigma^c\)\nb-ground states on \(\Cst(G_T;\omega)\), where \(\omega\) is a \(2\)-cocycle. We then consider cocycles arising from potentials and
describe the corresponding ground states for the twisted
\(\Cst\)-algebra \(\Cst(G_T;\omega)\) (see Corollary~\ref{coro-ground-st-DR-gpd}).

We next apply our result to twisted \(\Cst\)\nb-algebras of higher-rank graphs. For the \emph{preferred dynamics} introduced in~\cite{an-Huef-Laca-Raeburn-Sims-KMS-states-k-graph-periodicity}, we obtain a particularly simple characterization of the ground states. Let
\(\Lambda\) be a strongly connected row-finite \(k\)-graph with no source. For \(i=1,\ldots,k\), let
\(A_i\in\Mat_{\Lambda^0}\) be the matrix given by \(A_i(v,u)=|v\Lambda^{e_i}u|\) and let \(\rho(A_i)\) denotes the spectral radius of \(\rho(A_i)\).
Then the \(\sigma^c\)\nb-ground state space of the twisted
\(\Cst\)\nb-algebra \(\Cst(\Lambda;\omega)\) is nonempty precisely when
\(\rho(A_i)=1\) for all \(i=1,2,\cdots,k\).
In this case, every state on \(\Cst(\Lambda;\omega)\) is a
\(\sigma^c\)\nb-ground state. On the other hand, if
\(\rho(A_i)\neq 1\) for at least one \(i\), then there are no
\(\sigma^c\)\nb-ground states (see Theorem~\ref{thm-ground-k-graph}). 
 
\smallskip

\paragraph{\itshape Structure of the article.} 
Section~\ref{sec-prelim} contains the
necessary background on groupoid \(\Cst\)\nb-algebras, Fell bundles, the
representation theory of Fell bundles, and ground states.
In Section~\ref{sec-induc-state}, we construct a \(\Cst\)\nb-correspondence
from \(\Cst(G;\A)\) to
\(\Cst(G(Z);\A|_{G(Z)})\) (see Theorem~\ref{thm-C-st-corr}) and use it to induce states
from \(\Cst(G(Z);\A|_{G(Z)})\) to \(\Cst(G;\A)\)
(see Theorem~\ref{thm-ind-state}). Section~\ref{sec-main-ground-st} contains the main
result of the paper, which characterize the
\(\sigma^c\)-ground states on \(\Cst(G;\A)\) in terms of the states on
\(\Cst(G(Z);\A|_{G(Z)})\) (see Theorem~\ref{thm-main-ground-st}).
In Section~\ref{sec-crystal}, we study the crystallization of Fell bundle
\(\Cst\)\nb-algebras and identify the crystal with the \(\Cst\)\nb-algebra
associated to the Fell bundle over the boundary groupoid. Finally, in Section~\ref{sec-application}, we apply
our main result to several classes of examples, including twisted groupoid
\(\Cst\)\nb-algebras, twisted \(\Cst\)\nb-algebras of Deaconu--Renault groupoids,
twisted \(\Cst\)\nb-algebras of higher-rank graphs, and groupoid crossed
products.

\section{Preliminaries}
\label{sec-prelim}

\subsection{Groupoid \(\Cst\)-algebras}
We refer the reader to~\cite{Renault1980A-gpd-appr-to-cst-alg, Williams2019A-toolkit-for-gpd-alg} and~\cite{Sims-Szabo-Williams2020Book-Operator-alg-Dynamical-system-gpd-crossed-prod-Rokhlin-dim} for the standard materials for groupoids and their \(\Cst\)\nb-algebras. A groupoid \(G\) is a small category with every arrow is invartable. The unit space of the groupoid is denoted by \(\base\) and the set of all composable pairs is denoted by \(G^{(2)}\). The multiplication is a map from \(G^{(2)} \to G\). The range and source maps \(r,s \colon G\rightrightarrows \base\) are given by \(r(\gamma) = \gamma^{-1}\gamma\) and \(s(\gamma) = \gamma \gamma^{-1}\). The groupoid \(G\) is called \emph{locally compact Hausdorff} if \(G\) is a locally compact Hausdorff topological space such that all the structure maps are continuous. The groupoid \(G\) is called \emph{\'etale} if the range (or equivalently, the source) map is a local homeomorphism. An open set \(U\subset G\) is called \emph{bisection} if \(r(U), s(U)\subseteq \base\) are open and \(r|_U\colon U\to r(U)\) and \(s|_{U} \colon U\to s(U)\) are homeomorphisms. An \'etale groupoid has a basis consisting of bisections.

For a unit \(x\in \base\), \(G^x\defeq r^{-1}(x)\) and \(G_x\defeq s^{-1}(x)\) are called the range and source fibres, respectively. The intersection \(G^x_x\defeq G_x\cap G^x\) is a group, called the isotropy group at \(x\). For a subset \(Y\subset \base\), we denote the set \(\{\gamma \in G: s(\gamma) \in Y\}\) by \(G_Y\) and \(\{\gamma \in G: r(\gamma) \in Y\}\) by \(G^Y\). The subgroupoid \(G(Y)\defeq \{\gamma \in G: s(\gamma), r(\gamma) \in Y\}\) is the restriction of \(G\) on \(Y\).

Let \(\mu\) be a probability measure on \(\base\). Then \(\mu\) induces two measures \(\nu\) and \(\nu^{-1}\) on \(G\) given by 
\[
 \nu(f) =\int_{\base} \sum_{\gamma\in G^x} f(\gamma) \dd(\mu)(x) \quad \textup{and} \quad  \nu^{-1}(f) =\int_{\base} \sum_{\gamma\in G_x} f(\gamma) \dd(\mu)(x)
\]  
for \(f\in \Cc(G)\). The measure \(\mu\) is called \emph{quasi-invariant} if \(\nu\) and \(\nu^{-1}\) are equivalent. The Radon--Nikodym derivative \(\frac{\dd\nu}{\dd\nu^{-1}}\) is a \(1\)\nb-cocycle on \(G\), called the modular function (see~\cite[Page 23]{Renault1980A-gpd-appr-to-cst-alg}).

Let \(G\) be an \'etale groupoid. A \(2\)-cocycle \(\omega\) is a continuous map \(\omega \colon G^{(2)} \to \mathbb{T}\) satisfying
\[
\omega(\gamma, s(\gamma)) =\omega(r(\gamma), \gamma) = 1 \quad \textup{and} \quad \omega(\alpha, \beta\gamma) \omega (\beta, \gamma) = \omega(\alpha, \beta) \omega(\alpha\beta, \gamma)
\]
for all composable triples \((\alpha, \beta, \gamma)\). The set of all \(2\)-cocycle is denoted by \(Z^2(G; \mathbb{T})\). Fix \(\omega \in Z^2(G;\mathbb{T})\). The collection of all compactly supported complex valued functions on \(G\) is denoted by \(\Cc(G;\omega)\). Define convolution and involution on \(\Cc(G;\omega)\) by 
\[
 f*g(\gamma) = \sum_{\eta \in G^{r(\gamma)}} \omega(\eta, \eta^{-1}\gamma)f(\eta)g(\eta^{-1}\gamma) \quad \textup{and} \quad f^*(\gamma) = \overline{\omega(\gamma, \gamma^{-1})f(\gamma^{-1})}.
\]
 The enveloping \(\Cst\)\nb-algebra of \(\Cc(G;\omega)\) is called the \emph{twisted groupoid \(\Cst\)-algebra} and denoted by \(\Cst(G;\omega)\).
\subsection{Fell bundle \(\Cst\)-algebras}
We recall briefly the definition of upper semicontinuous bundle of Banach spaces over a locally compact Hausdorff space from~\cite[Appendix C]{Williams2007Crossed-Prod-Book} and~\cite{Fell1988Representation-of-Star-Alg-Banach-bundles}. An \emph{upper semicontinuous bundle of Banach spaces} (or \emph{upper semicontinuous Banach bundle}) is topological space \(\A\) together with an open map \(p\colon \A \to X\), \(\A\) is called the total space and \(X\) is the base space. For each \(x\in X\), the fibre \(\A_x\defeq p^{-1}(x)\) has a complex Banach space structure such that the the function \(\A\to \R\) by \(a\mapsto \norm{a}\) is supper semicontinuous and the scalar multiplication \(\C\times \A \to \A\) is continuous.

In this article, we only consider upper semicontinuous bundle over a locally compact Hausdorff second countable space \(X\).
For such a bundle is said to have \emph{enough sections} if for any \(x\in X\) and \(a\in \A_x\), there exists a continuous compactly supported section \(f\) such that \(f(x) =a\). If the base space \(X\) is locally compact Hausdorff then the bundle \(\A\) have enough sections by~\cite{Hofmann1977Bundles-and-sheaves-equi-in-Cat-Ban}.

For an upper semicontinuous Banach bundle \(p\colon \A\to X\), the collection of continuous compactly supported sections and denoted by 
\(\Cc(X;\A)\) and the collection of continuous sections that vanishes at infinity is denoted by \(\Contz(X; \A)\). Recall the definition of upper semicontinuous bundle of \(\Cst\)\nb-algebras and \(\Contz(X)\)-algebras from~\cite[Appendix C]{Williams2007Crossed-Prod-Book}. Recall from~\cite[Theorem C.26]{Williams2007Crossed-Prod-Book} that there is an one-to-one correspondence between the \(\Contz(X)\)-algebras and the  section algebras of upper semicontinuous bundle of \(\Cst\)\nb-algebras over \(X\). The fibre of a \(\Contz(X)\)\nb-algebra \(A\) at \(x\in X\) is denoted by \(A(x)\) or \(\A_x\) where \(\A\) is the bundle associated to \(A\).

Let \(G\) be a locally compact Hausdorff \'etale groupoid and \(p\colon\A\to G\) an upper semicontinuous bundle of Banach spaces. Let 
\[
 \A^{(2)}\defeq \{(a,b)\in \A\times\A : s(p(a)) = r(p(b))\}.
\]

\begin{definition}[Fell bundle~\cite{Muhly-Williams2008Equivalence-and-disinte-thm-Fell-bundle}]
	\label{def:fell-bundle}
	A \emph{Fell bundle} over an \etale\ groupoid~\(G\) is an upper
	semicontinuous bundle of Banach spaces \(p\colon \A\to G\) equipped
	with a continuous \emph{multiplication}  from \( \A^{(2)}\to \A, (a,b) \mapsto ab\)
	and an \emph{involution} from \(\A\to \A, a\mapsto a^*\) for \(a\in \A\)
	which satisfy the following axioms:
	\begin{enumerate}[leftmargin = *]
		\item\label{item:Fell-1} \(p(ab) =p(a)p(b)\) for all
		\((a,b) \in \A^{(2)}\);
		\item\label{item:Fell-2} \(p(a^*)=p(a)^{-1}\)
		
		\item \label{item:Fell-4} for each \(x\in \base\), \(\A_x\) is a \(\Cst\)\nb-algebra with respect to the inherited
		multiplication and involution on \(\A_x\);
		
		\item \label{item:Fell-5} for \(\gamma\in G\), \(\A_\gamma\) is a \(\A_{r(\gamma)}\)-\(\A_{s(\gamma)}\)-imprimitivity bimodule equipped
		with the inherited actions and inner products given by
		\[
		{}_{\A_{r(\gamma)}}\Linpro{a}{b}\defeq ab^* \quad \textup{and} \quad \inpro{a}{b}_{\A_{s(\gamma)}} \defeq a^*b.
		\]
\end{enumerate}
\end{definition}

Let \(p\colon \A \to G\) be a Fell bundle over a locally compact Hausdorff \etale\ groupoid~\(G\). We assume the Fell bundle is \emph{saturated} in the sense that 
\[
\A_{\gamma\A_{\eta}}\defeq \textup{span}\{ab: a\in \A_{\gamma} \textup{ and } b\in \A_{\eta}\}
\]
 is dense in \(\A_{\gamma\eta}\). Let \(\Cc(G;\A)\) denote the \(^*\)-algebra with the following convolution and involution operations:
\[
f*g(\gamma) = \sum_{\alpha\beta=\gamma} f(\alpha)g(\beta) \quad \text{and} \quad f^*(\gamma) = f(\gamma^{-1})^*
\]
where \(f, g \in \Contc(G;\A)\).
Define \(I\)\nb-norm on \(\Cc(G;\A)\) by 
\[
\norm{f}_I =\max\Big\{\sup_{x\in \base} \sum_{\gamma\in G^x} \norm{f(\gamma)}, \sup_{x\in \base} \sum_{\gamma\in G_x} \norm{f(\gamma)} \Big\}.
\]
A \emph{representation} of \(\Cc(G;\A)\) is a \(^*\)\nb-homomorphism \(\pi\colon \Cc(G; \A) \to \Bound(H)\). The representation \(\pi\) is called \emph{\(I\)\nb-norm bounded} if \(\norm{\pi(f)} \leq \norm{f}_I\) for all \(f\in \Cc(G;\A)\). The \emph{universal norm} on \(\Cc(G;\A)\) is defined by 
\[
 \norm{f} \defeq \sup\{\norm{\pi(f)} : \pi \textup{ is a \(I\)-norm bounded representation}\}.
\]
The completion of \(\Cc(G;\A)\) with respect to the universal norm is called the \emph{full \(\Cst\)-algebra} of the Fell bundle and denoted by \(\Cst(G;\A)\).

For a subset \(K \subseteq G\), we denote the restriction of the Fell bundle \(\A\) over \(K\) as \(\A|_{K}\). If \(K\) is a subgroupoid of \(G\), then \(\A|_K\) is again a Fell bundle over \(K\). The space of compactly supported continuous sections of the restricted bundle \(\A|_K \to K\) is denoted by \(\Cc(K;\A|_{K})\).
\subsection{Representations of Fell bundles}
We now recall the necessary background of representations of Fell bundles and their integrated representation  from~\cite{Muhly-Williams2008Equivalence-and-disinte-thm-Fell-bundle} which will require to prove our main result in Section~\ref{sec-main-ground-st}.  

Let \(p\colon \A\to G\) be a Fell bundle over a locally compact Hausdorff groupoid \(G\) and \(\Hilm \to \base\) a Borel Hilbert bundle over the unit space \(\base\). Consider the groupoid 
\[
\textup{End}(\Hilm) \defeq \{(x,T,y): x,y\in \base \textup{ and } T\in \Bound(\Hilm_y,\Hilm_x)\}.
\] 
\begin{definition}[{\cite[Defintion 4.5]{Muhly-Williams2008Equivalence-and-disinte-thm-Fell-bundle}}]
	A map \(\hat{\pi}\colon \A \to \textup{End}(\Hilm)\) is called \(^*\)\nb-functor if for \(a\in \A\),
\(\hat{\pi}(a) = (r(p(a)), \pi(a), s(p(a)))\) where 
\(\pi(a) \colon \Hilm_{s(p(a))} \to \Hilm_{r(p(a))}\) is a bounded linear operator such that
\begin{enumerate}
	\item \(\pi(\lambda a+b)=\lambda \pi(a)+\pi(b)\) if
	\(a,b \in \A_\gamma\) for \(\gamma \in G\);
	\item \(\pi(ab)=\pi(a)\pi(b)\) if \((a,b)\in \A^{(2)}\);
	\item \(\pi(a^*)=\pi(a)^*\) for \(a\in \A\).
\end{enumerate}
\end{definition}
\noindent A \emph{strict representation} of the Fell bundle \(\A\) is \((\mu,\Hilm, \hat{\pi})\) consisting of a quasi-invariant
probability measure \(\mu\) on \(G^{(0)}\), a Borel Hilbert bundle
\(\Hilm\to \base\) and a \(^*\)-functor \(\hat{\pi}\).
Let
\(L^2(\Hilm,\mu)\) be the direct integral
\(\int_{\base}^{\oplus} \Hilm_x\,\dd\mu(x)\) of square-integrable
section of the bundle~\(\Hilm\).  Let \(\Delta_\mu\) be the modular function
associated with the quasi-invariant measure~\(\mu\). This strict
representation integrates to an \(I\)-norm bounded representation
\(L\) of \(\Contc(G;\A)\) on \(L^2(\Hilm,\mu)\) (see~\cite[Proposition~4.10]{Muhly-Williams2008Equivalence-and-disinte-thm-Fell-bundle}) given by
\begin{equation}\label{equ-int-dis-rep}
\inpro{\eta}{L(f)\zeta} = \int_{G^{(0)}} \biggr(\sum_{\gamma\in G^x}
\inpro{\eta(r(\gamma))}{\pi(f(\gamma))\zeta(s(\gamma))}_{\Hilm_{r(\gamma)}}
\Delta_\mu(\gamma)^{-\frac{1}{2}} \biggr) \mathrm{d} \mu(x)
\end{equation}
for \(f\in \Cc(G;\A)\) and \(\zeta,\eta \in L^2(\Hilm,\mu)\).  The
disintegration theorem~\cite[Theorem
4.13]{Muhly-Williams2008Equivalence-and-disinte-thm-Fell-bundle} established a strong converse of this
integration
process~\cite[Proposition~4.10]{Muhly-Williams2008Equivalence-and-disinte-thm-Fell-bundle}. Using vector-valued integral, the
 Equation~\eqref{equ-int-dis-rep} can be written as
\begin{equation}\label{equ-vector-valued-int-dis-rep}
L(f)\zeta(x) =  \sum_{\gamma\in G^x}
\pi(f(\gamma))\zeta(s(\gamma)) 
~\Delta_\mu(\gamma)^{-\frac{1}{2}}\,
\end{equation}
for \(x\in\base\), \(\zeta \in L^2(\Hilm, \mu)\) and \(f\in \Cc(G;\A)\).

\subsection{KMS states and ground states}\label{subsec-KMS-ground-st}
We refer the reader to~\cite{Bratteli-Robinson1981Oper-alg-Quan-sta-mech-part-2} for the standard background on KMS states and ground states. Let \(A\) be a \(\Cst\)\nb-algebra with a one-parameter group of automorphisms \(\sigma = (\sigma_t)_{t\in \R}\). In other words \((A, \sigma)\) is a real \(\Cst\)-dynamical system.  An element
\(a\in A\) is called \emph{analytic} if the map \(t\mapsto \sigma_t(a)\) from \(\R\to A\) 
is the restriction of an analytic function \(\mathbb{C}\to
A\). The collection of all analytic elements of \(A\) forms a norm-dense \(^*\)-subalgebra of \(A\). Let \(\beta \in \R\). A state \(\varphi\) on \(A\) is called a \(\KMS\) state if 
\[
 \varphi(ab) = \varphi(b\sigma_{i\beta}(a))
\] 
for all analytic elements \(a,b\in A\).

Let \(p\colon \A \to G\) be a Fell bundle over an \etale\ groupoid \(G\).
Let \(c\colon G\to \R\) be a 1-cocycle, i.e., a continuous
groupoid homomorphism. Such a cocycle induces a real dynamics
\(\sigma^c \colon \R \to \textup{Aut}(\Cst(G;\A))\) defined by
\begin{equation}\label{equ:real-dyna}
	\sigma^c_t(f)(\gamma)= \mathrm{e}^{itc(\gamma)}f(\gamma)
\end{equation}
for \(t\in \R, \gamma \in G\) and
\(f \in \Cc(G;\A)
\).

The characterization of KMS\(_\beta\) states on the groupoid \(\Cst\)\nb-algebra \(\Cst(G)\) in terms of quasi-invariant measures and a field of states ware obtained by Renault~\cite[Proposition II.5.4]{Renault1980A-gpd-appr-to-cst-alg} and subsequently by Neshveyev~\cite[Theorem 1.3]{Neshveyev2013KMS-states}. This result was generalized to \emph{singly generated} Fell bundles over \'etale groupoids by Afsar--Sims~\cite[Theorem 3.4]{Afsar-Sims2021KMS-state-on-Cst-alg-Fell-bundle-over-gpd}.
For general saturated Fell bundles, the result is as follows.

\begin{theorem}[{\cite[Theorem 5.20]{Holkar-Hossain-2024-KMS-states}}]
	\label{thm-KMS-state}
	Let \(p\colon \A \to G\) be a Fell bundle over a 
	locally compact Hausdorff second countable {\'e}tale groupoid
	\(G\). Let \(\beta\in \R\). Then, there is a bijetive
	correspondence between the \KMS\ states on \((\Cst(G;\A), \sigma^c)\)
	and pairs \((\mu, \{\varphi_x\}_{x})\) consisting of a probability
	measure \(\mu\) on \(G^{(0)}\) and a \(\mu\)\nb-measurable
	field of states \(\{\varphi_x\}_{x}\) on
	\(\Cst(G^x_x;\A|_{G^x_x})\) with \(\A_x\) contained in the
	centraliser of~\(\varphi_x\) such that the following holds:
	\begin{enumerate}
		\item \(\mu\) is quasi-invariant probability
		measure with Radon--Nikodym derivative~\(\mathrm{e}^{-\beta c}\).
		\item  For \(\mu\)-a.e. \(x\in \base\) and
		\(\gamma \in G^x_x\), \(\eta \in G_x\), \(a\in \A_\gamma\)
		and \(\xi \in \A_\eta\) the following equality holds:
		\begin{equation*}
			\varphi_{s(\eta)} \big(a \xi^* \xi \cdot
			\delta_{\gamma}\big) = \varphi_{r(\eta)}\big(\xi   a \xi^* \cdot
			\delta_{\eta \gamma \eta^{-1}}\big).
		\end{equation*}
	\end{enumerate}
	The state corresponding to the pair \((\mu, \{\varphi_x\}_{x})\) is given by 
	\[
	 \varphi(f) =\int_{G^{(0)}} \sum_{\gamma \in G^x_x} \varphi_x(f(\gamma)\delta_{\gamma}) \dd\mu(x)
	\]
	for \(f\in \Cc(G;\A)\).
\end{theorem}

\noindent 
We wish to obtain a characterization analogous to Theorem~\ref{thm-KMS-state} for \(\sigma^c\)-ground states on the Fell bundle \(\Cst\)\nb-algebra \(\Cst(G;\A)\). 

A state \(\varphi\) on \((A, \sigma)\) is called a ground state (or \(\sigma\)-ground state) if for every \(a\in A\) and every analytic element \(b\) the function 
\[
 z\mapsto \varphi(a\sigma_z(b))
\]
 is bounded on \(\{z\in \C : \textup{Im}(z)\geq 0\}\) (see~\cite[Section 5.3]{Bratteli-Robinson1981Oper-alg-Quan-sta-mech-part-2}). Such a state is always \(\sigma\)\nb-invariant. A state on \(A\) is called a KMS\(_\infty\) state if it is weak\(^*\)-limit of a net of KMS\(_{\beta_i}\) states \((\varphi_i)_i\) such that \(\beta_i\to \infty\). Any KMS\(_\infty\) state is a ground state (see~\cite[Proposition 5.3.23]{Bratteli-Robinson1981Oper-alg-Quan-sta-mech-part-2}). However, the converse does not holds in general. (See~\cite[Corollary 1.8]{Laca-Larsen-Neshveyev-Ground-states}, \cite{Laca-Raeburn-Phase-transation-of-Top-alg-affi-sem-gp-n} and Remark~\ref{rmk-Infinity}).

Let \(c\) be a real-valued \(1\)-cocycle of \(G\). Recall the boundary set \(Z\) from~\cite{Renault1980A-gpd-appr-to-cst-alg},~\cite{Laca-Larsen-Neshveyev-Ground-states} of the cocycle \(c\) 
\begin{equation}\label{eq-boundary-set}
Z = \{x\in \base: c\leq 0 \textup{ on } G^x\} = \{x\in \base: c\geq 0 \textup{ on } G_x\}.
\end{equation}
Renault~\cite{Renault1980A-gpd-appr-to-cst-alg} denote this boundary set by `\(\textup{Min}(c)\)', however we follow the notation of~\cite{Laca-Larsen-Neshveyev-Ground-states} and denote the boundary set by \(Z\). The boundary set may be empty, but when it is nonempty, it plays a crucial role
to characterize the \(\sigma^c\)-ground states for the Fell bundle algebra \(\Cst(G;\A)\). Consider the subgroupoid \(G(Z)\defeq  r^{-1}(Z)\cap s^{-1}(Z) \) of \(G\). Lemma~1.3 of~\cite{Laca-Larsen-Neshveyev-Ground-states} says that if the kernel groupoid \(c^{-1}(0)\) of the cocycle is \'etale, then \(G(Z)\) is also \'etale. We denote \(s^{-1}(Z)\) by the notation \(G_Z\).

\noindent The following result is standard, but we include a proof for completeness.
\begin{lemma}\label{lem-closed-of-boundary}
	Let \(G\) be a locally compact Hausdorff \'etale groupoid and \(c\colon G \to \R\) a \(1\)\nb-cocycle. Then the boundary set 
	\[
	Z =\{x\in \base: c(\gamma) \leq 0 \textup{ for all } \gamma \in G^x\}
	\]
	is closed and consequently, \(G(Z)\) is a closed subgroupoid of \(G\).
\end{lemma}
\begin{proof}
	Let \(x_i\to x\) for \(x_i\in Z\). Since \(r\colon G\to \base\) is a local homeomorphism, for \(\gamma\in G^x\) we choose an open bisection \(U\) containing \(\gamma\) such that \(r|_{U} \colon U\to r(U)\) is a homeomorphism. Since \(x\in r(U)\) and \(x_i\to x\),  we have \(x_i\in r(U)\) for sufficiently large \(i\). Therefore, \(\gamma_i\defeq (r|_U)^{-1}(x_i) \to \gamma\). Since \(x_i\in Z\), \(c(\gamma_i)\leq 0\). Thus, \(c(\gamma) \leq 0\) and hence \(x\in Z\). Therefore, \(Z\) is a closed subset of \(\base\). 
	
Since \(Z\) is closed, both \(r^{-1}(Z)\) and \(s^{-1}(Z)\) are closed in \(G\). Consequently,
\(
G(Z)=r^{-1}(Z)\cap s^{-1}(Z)\)
is a closed subgroupoid of \(G\).
\end{proof}
\noindent We end this section with the following statement of an approximation identity of Hilbert module.
\begin{lemma}[{\cite[Page 5]{Lance1995Hilbert-modules}}]\label{lem-app-id-hilbert-mod}
Let \(X\) be a right Hilbert \(A\)\nb-module and \((u_\lambda)_{\lambda}\) an approximate identity of \(A\). Then \(\norm{x u_\lambda -x} \to 0\) as \(\lambda \to \infty\).
\end{lemma}

\section{Induction of states}\label{sec-induc-state}
 
 In this section, we develop an induction procedure for states from the \(\Cst\)\nb-algebra of Fell bundle over the boundary groupoid \(G(Z)\) to the \(\Cst\)\nb-algebra of the given Fell bundle \(p\colon \A\to G\). The construction is based on the groupoid equivalence between the boundary groupoid \(G(Z)\) and the imprimitivity groupoid \(H^G\). This induction will be useful in the proof of our main result in Section~\ref{sec-main-ground-st}.

Let \(p\colon \A \to G\) be a Fell bundle over a locally compact Hausdorff \'etale groupoid~\(G\). Let \(Z\) be the boundary set defined in Equation~\eqref{eq-boundary-set}. Then the reduction \(G(Z) = r^{-1}(Z)\cap s^{-1}(Z)\) is a closed subgroupoid of \(G\). Throughout this section, we denote \(G(Z)\) by \(H\) and assume that \(H\) is \'etale. Let \(p\colon \A|_{H} \to H\) be the restricted Fell bundle to the subgroupoid \(H\). We first briefly recall the \emph{imprimitivity groupoid} \(H^G\) from~\cite[Page 47]{Williams2019A-toolkit-for-gpd-alg}. Consider the free and proper \(H\)-space 
\[
 G_Z\times_{s, Z,s}G_Z =\{(\gamma, \eta) \in G_Z\times G_Z : s(\gamma) =s(\eta)\}.
\]
The orbit space \(H^G\defeq (G_Z\times_{s, Z,s}G_Z)/H\) is a locally compact Hausdorff groupoid and the groupoid structure is as follows: we denote the orbit of \((\gamma, \eta)\) by \([\gamma, \eta]\). The range and source maps \(r,s \colon  H^G \rightrightarrows G_Z/H\) are given by \(r([\gamma, \eta]) = \gamma H\) and  \(s([\gamma, \eta]) = \eta H\). We say \([\gamma, \eta]\) and \([\omega, \zeta]\) are composable if \(\eta H = \omega H\) and the operations are given by 
\[
 [\gamma, \eta] [\omega, \zeta] = [\gamma, \zeta] \quad \textup{and} \quad [\gamma, \eta]^{-1} = [\eta, \gamma].
\]
Then, \(G_Z\) is a \((H^G, H)\)-groupoid equivalence.

Let \(p\colon \A\to G\) be a Fell bundle. Define a continuous groupoid homomorphism \(\rho \colon H^G \to G\) by \(\rho([\gamma, \eta])  =\gamma\eta^{-1}\). The \emph{pull back} bundle 
\[
 \rho^*\A \defeq \{(a,[\gamma, \eta]) : [\gamma, \eta] \in H^G , a\in \A \textup{ and } p(a)= \rho([\gamma, \eta]) \}
\]
is a Fell bundle over \( H^G\) with the bundle map \(\rho^*p(a,[\gamma, \eta]) = [\gamma, \eta]\). 
The restricted bundle \(p\colon \A|_{G_Z} \to G_Z\) implements an equivalence between the Fell bundles \(\rho^*p\colon \rho^*\A \to H^G\) and \(p\colon \A|_{G(Z)} \to G(Z)\) in the sense of Muhly--Williams~\cite[Definition 6.1]{Muhly-Williams2008Equivalence-and-disinte-thm-Fell-bundle}. The left action of \(\rho^*\A\) on \(\A|_{G_Z}\) given by 
\((a,[\gamma, \eta])b = ab \) if \(\eta H = r(p(b))\); the right action of \(\A|_H\) on \(\A|_{G_Z}\) is given by \(b\cdot e =be\). The left inner product \(\Linpro{\cdot}{\cdot} \colon \A|_{G_Z}*_s\A|_{G_Z} \to \rho^*\A\) and the right inner product \(\inpro{\cdot}{\cdot} \colon \A|_{G_Z}*_r\A|_{G_Z} \to \A|_H\) are given by 
\[
 \Linpro{e}{f} \defeq  (ef^*, [p(e), p(f)]) \quad \textup{and} \quad \inpro{e}{f} \defeq e^*f
\]
where \(\A|_{G_Z}*_s\A|_{G_Z} \defeq \{(e,f) \in \A|_{G_Z} \times \A|_{G_Z} : s(p(e)) = s(p(f))\}\) and similarly we can define \(\A|_{G_Z}*_r\A|_{G_Z}\). Then \(\Cc(G_Z;\A|_{G_Z})\) is a pre-imprimitivity bimodule between \(\Cc(H^G;\rho^*\A)\) and \(\Cc(H;\A|_{H})\) with actions and inner products are given by 
\begin{align}
	F*'\phi(\gamma) &= \sum_{\eta \in G_{s(\gamma)}} F([\gamma, \eta]) \phi(\eta)\label{eq-equivalence-1}\\
	\phi*_rg(\gamma) &= \sum_{h\in G^{s(\gamma)}_{Z}} \phi(\gamma h) g(h^{-1})  \label{eq-equivalence-2}\\
	\inpro{\phi}{\psi}(h) &= \sum_{\gamma \in G_{r(h)}} \phi(\gamma)^*\psi(\gamma h) \label{eq-equivalence-3}\\
	\Linpro{\phi}{\psi}([\gamma, \eta]) &= \sum_{h\in G^{s(\gamma)}_Z} \phi(\gamma h) \psi(\eta h)^* \label{eq-equivalence-4}
\end{align}
where \(F\in \Cc(H^G;\sigma^*\A)\), \(\phi, \psi \in \Cc(G_Z;\A|_{G_Z})\) and \(g\in \Cc(H;\A|_{H})\). The completion \(X\) of \(\Cc(G_Z;\A|_{G_Z})\) is a \(\Cst(H^G;\sigma^*\A)\)-\(\Cst(H;\A|_{H})\)-imprimitivity bimodule (see~\cite[Section 2.1]{Ionescu-Williams-Irre-Ind-rep-Fell-bund} for details).

Our first goal is to define a left action of \(\Cc(G;\A)\) on the right Hilbert \(\Cst(G(Z);\A|_{G(Z)})\)-module and complete it to a \(\Cst\)-correspondence from \(\Cst(G;\A) \to \Cst(G(Z);\A|_{G(Z)})\). To achieve this we define the left action by
\begin{equation}\label{eq-left-action}
	f*_l\phi(\gamma) = \sum_{\eta\in G^{r(\gamma)}} f(\eta)\phi(\eta^{-1}\gamma)
\end{equation} 
for \(f\in \Cc(G;\A)\) and \(\phi\in \Cc(G_Z;\A|_{G_Z})\).
\begin{lemma}\label{left-aaction-adjointable}
The left action \(*_l\) given by Equation~\eqref{eq-left-action} satisfies
\[
 \inpro{f*_l\zeta}{\eta} =\inpro{\zeta}{f^**_l\eta}
\] 
for \(f\in \Cc(G;\A)\) and \(\zeta, \eta\in \Cc(G_Z;\A|_{G_Z})\).
\end{lemma}
\begin{proof}
For \(h\in G(Z)\), we have
\begin{align*}
\inpro{f*_l\zeta}{\eta} (h) &= \sum_{\gamma \in G_{r(h)}} ((f*_l\zeta)(\gamma))^*\eta(\gamma h) = \sum_{\gamma \in G_{r(h)}} \sum_{\tau\in G^{r(\gamma)}} \bigl( f(\tau)\zeta(\tau^{-1}\gamma)\bigr)^*\eta(\gamma h)\\
&=\sum_{\gamma \in G_{r(h)}} \sum_{\tau\in G^{r(\gamma)}} \zeta(\tau^{-1}\gamma)^*f(\tau)^*\eta(\gamma h).
\end{align*}
Using the change of variable \(\gamma \mapsto \tau\gamma\), the last term becomes
\[
 \sum_{\gamma \in G_{r(h)}} \sum_{\tau\in G_{r(\gamma)}} \zeta(\gamma)^*f(\tau)^*\eta(\tau\gamma h).
\]
Since \(\tau \mapsto \tau^{-1}\) is a  bijection from \(G_{r(\gamma)} \to G^{r(\gamma)}\) using the change of variable \(\tau \mapsto \tau^{-1}\) in the last term, we obtain
\begin{multline*}
\sum_{\gamma \in G_{r(h)}} \sum_{\tau\in G^{r(\gamma)}} \zeta(\gamma)^*f(\tau^{-1})^*\eta(\tau^{-1}\gamma h) = \sum_{\gamma \in G_{r(h)}} \sum_{\tau\in G^{r(\gamma)}} \zeta(\gamma)^*f^*(\tau)\eta(\tau^{-1}\gamma h)\\
=\sum_{\gamma \in G_{r(h)}} \zeta(\gamma)^* f^**_l\eta(\gamma h) = \inpro{\zeta}{f^**_l\eta}.
\end{multline*}
\end{proof}
\begin{lemma}\label{lem-left-act-comp}
The left action of \(\Cc(G;\A)\) on the pre-Hilbert module \(\Cc(G_Z;\A|_{G_Z})\) defined by Equation~\eqref{eq-left-action} is compatable with the module operations of \(X\), i.e., we have 
\begin{align}
	g*_l(\zeta+\eta) &= g*_l\zeta +g*_l\eta \label{eq-induction-1}\\
	(g+k)*_l\zeta &= g*_l\zeta + k*_l\zeta \label{eq-induction-2}\\
	(g*k)*_l\zeta &= g *_l(k*_l\zeta) \label{eq-induction-3}\\
	(g*_l\zeta)*_rf &= g*_l(\zeta*_rf) \label{eq-induction-4}
\end{align}
for \(g,k\in \Cc(G;\A)\), \(\zeta, \eta \in \Cc(G_Z;\A|_{G_Z})\) and \(f\in \Cc(G(Z);\A|_{G(Z)})\).
\end{lemma}
\noindent The proof of Equations~\eqref{eq-induction-1} and~\eqref{eq-induction-2} are follows from the definition, while the proof of Equations~\eqref{eq-induction-3} and~\eqref{eq-induction-4} are analogous to that of~\cite[Lemma 3.6]{Holkar-Hossain-2024-KMS-states}, hence we omit the details.

Our next goal is to show that the left action \(*_l\) of \(\Cc(G;\A)\) on \(X\) given by Equation~\eqref{eq-left-action} is nondegenerate and can be extended to \(\Cst(G;\A)\). To prove the nondegeneracy we need some preparation.

\begin{lemma}\label{lem-ind-dens}
	\begin{enumerate}
		\item The span of \(\{f*_l\zeta : f\in \Cc(G;\A) \textup{ and } \zeta \in \Cc(G_Z;\A|_{G_Z})\}\) is dense in \(\Cc(G_Z;\A|_{G_Z})\) in the inductive limit topology.
		\item Let \(\zeta, \eta\in \Cc(G_Z;\A|_{G_Z})\). Then the map \(f\mapsto \inpro{\zeta}{f*_l\eta}\) from \(\Cc(G;\A) \to \Cc(G(Z);\A|_{G(Z)})\) is continuous in the inductive limit topology.  
	\end{enumerate}
\end{lemma}
\begin{proof}
\noindent (1). Let \((f_n)_{n\in \N}\) be an approximate identity for \(\Cst(G;\A)\) such that each \(f_n \in \Cc(\base;\A|_{\base})\) (such approximate identity always exists by~\cite[Proposition 2.23]{Holkar-Hossain-2024-KMS-states}). Now for \(\zeta \in \Cc(G_Z;\A|_{G_Z})\),  we have 
\[
 f_n*_l\zeta(\gamma) = \sum_{\alpha \in G^{r(\gamma)}} f_n(\alpha) \zeta(\alpha^{-1}\gamma) = f_n(r(\gamma)) \zeta(\gamma)
\]
for \(\gamma \in G_Z\). Since \(\A_{\gamma}\) is a  full left Hilbert \(\A_{r(\gamma)}\)-module and \((f_n(r(\gamma)))_{n\in \N}\) is an approximate identity for \(\A_{r(\gamma)}\) (by~\cite[Lemma 2.14]{Holkar-Hossain-2024-KMS-states}), thus we have 
\[
 f_n*_l\zeta (\gamma) = f_n(r(\gamma)) \zeta(\gamma) \to \zeta(\gamma)
\]
for \(\gamma \in G_Z\) (by Lemma~\ref{lem-app-id-hilbert-mod}). Since \(\supp(f_n*_l\zeta) \subseteq \supp(\zeta)\), which is a compact set, we have 
\[
 \norm{f_n*_l\zeta-\zeta}_{\infty} = \max_{\gamma \in \supp(\zeta)} \norm{f_n(r(\gamma)) \zeta(\gamma) - \zeta(\gamma)} \to 0
\] 
as \(n\to \infty\). Therefore, \(f_n*_l\zeta \to \zeta\) in the inductive limit topology.

\noindent (2). The proof is similar to that of~\cite[Lemma 3.20]{Holkar-Hossain-2024-KMS-states} and hence omitted.
\end{proof}

To show that the left action \(*_l\) is bounded, let \(\varphi\) be a state on \(\Cst(G(Z); \A|_{G(Z)})\). Consider the Hilbert space completion of the inner product space \((X, \inpro{\cdot}{\cdot}_{\varphi})\), where \(\inpro{\zeta}{\eta}_{\varphi} \defeq \varphi(\inpro{\zeta}{\eta}) \) for \(\zeta,\eta\in X\). We denote this Hilbert space completion by~\(X_{\varphi}\).
Define the vector subspace \(Y\subseteq X_{\varphi}\) generated by \(\{f*_l\zeta : f\in \Cc(G;\A) \textup{ and } \zeta \in \Cc(G_Z;\A|_{G_Z})\}\). Lemma~\ref{lem-ind-dens}(1) shows that \(Y\) is dense in \(X\). Define a representation \(L \) of \(\Cc(G;\A)\) on \(Y\) by 
\[
L(f)(\zeta) = f*_l\zeta
\]
for \(f\in \Cc(G;\A)\) and \(\zeta \in \Cc(G_Z; \A|_{G_Z})\). We now show that \(L\) satisfies the three conditions required to be a pre-representation (see Definition~4.1 of~\cite{Muhly-Williams2008Equivalence-and-disinte-thm-Fell-bundle}) of \(\Cc(G;\A)\) on \(Y\).
\begin{enumerate}
	\item \(L\) is nondegenerate follows from  Lemma~\ref{lem-ind-dens}(1). 
	\item  Lemma~\ref{lem-ind-dens}(2) ensures that, the map \(f \mapsto \inpro{\zeta}{L(f)\eta}_{\varphi}\) is continuous in the inductive limit topology for \(f\in \Cc(G;\A)\) and \(\zeta, \eta \in \Cc(G_Z;\A|_{G_Z})\).
	\item Lemma~\ref{left-aaction-adjointable} ensures that the left action is adjointable.
\end{enumerate}
Therefore, the disintegration theorem for Fell bundles~\cite[Theorem 4.13]{Muhly-Williams2008Equivalence-and-disinte-thm-Fell-bundle}
allows us to extends \(L\) be a
nondegenerate bounded representation of \(\Cst(G;\A)\) on the Hilbert space \(X_{\varphi}\). Therefore,
\[
\varphi(\inpro{f*_l\zeta}{f*_l\zeta}) \leq \norm{f}^2_{\Cst(G;\A)} \varphi(\inpro{\zeta}{\zeta})
\]
for all \(f\in \Cc(G;\A)\) and \(\zeta \in \Cc(G_Z;\A|_{G_Z})\). Since \(\varphi\) was an
arbitrary state on \(\Cst(G(Z); \A|_{G(Z)})\), we have 
\[
\inpro{f*_l\zeta}{f*_l\zeta} \leq \norm{f}^2_{\Cst(G;\A)} \inpro{\zeta}{\zeta}.
\]
Therefore, the left action \(*_l\) of \(\Cc(G; \A)\) on  \(\Cc(G_Z;\A|_{G_Z})\)
extends to a
nondegenerate representation of \(\Cst(G;\A)\)
on \(X\).

\begin{theorem}\label{thm-C-st-corr}
Let \(p\colon \A \to G\) be a Fell bundle over a locally compact Hausdorff second countable \'etale groupoid \(G\) and let \(c\colon G \to \R\) be a \(1\)\nb-cocycle. Then \(\Cc(G_Z;\A|_{G_Z})\) completes to a \(\Cst\)\nb-correspondence 
\[
 X\colon \Cst(G;\A) \to \Cst(G(Z); \A|_{G(Z)})
\]
where \(Z\) is the boundary set given by Equation~\eqref{eq-boundary-set}.
\end{theorem}

\begin{proof}
The proof follows from the above discussion, Lemma~\ref{lem-left-act-comp} and the discussion on Page~\pageref{eq-equivalence-1}.
\end{proof}
\begin{remark}
	If \(Z=\{x\}\) for \(x\in \base\), then the last theorem recovers~\cite[Theorem 3.22]{Holkar-Hossain-2024-KMS-states}.
\end{remark}
Let \(\pi\colon \Cst(G(Z);\A|_{G(Z)}) \to \Bound(H)\) be a representation. Then~\cite[Proposition 2.66]{Raeburn-Williams1998Morita-equiv-continuous-trace-cst-alg} ensures that there is a induced representation  
\[
 \Ind(\pi) \colon \Cst(G;\A) \to \Bound(X\otimes H)
\]
by \(\Ind (\pi) (f) (\zeta\otimes h) = f*_l\zeta \otimes h\) for \(f\in \Cc(G;\A), \zeta \in \Cc(G_Z;\A|_{G_Z})\) and \(h\in H\). Recall from~\cite[Pages 33, 34 ]{Raeburn-Williams1998Morita-equiv-continuous-trace-cst-alg} that the inner product on the Hilbert space \(X\otimes H\) is given by 
\[
 \inpro{\zeta\otimes h}{\eta\otimes k} = \inpro{h}{\pi(\inpro{\zeta}{\eta})k}.
\]
Therefore, the above discussion and Theorem~\ref{thm-C-st-corr} gives the next corollary.
\begin{corollary}\label{cor-ind-rep}
Let \(p\colon \A \to G\) be a Fell bundle over a locally compact Hausdorff second countable \'etale groupoid \(G\) and let \(c\colon G \to \R\) be a \(1\)\nb-cocycle. Then every representation \(\pi\colon \Cst(G(Z);\A|_{G(Z)}) \to \Bound(H)\) induces a representation \(\Ind(\pi)\) of \(\Cst(G;\A)\) on the Hilbert space \(X\otimes H\). 
\end{corollary}

We fix an approximate identity \((u_n)_{n\in \N}\) for \(\Cst(G(Z);\A|_{G(Z)})\) such that \(u_n \in \Cc(Z;\A|_{Z})\) for all \(n\in \N\). Such approximate identity always exists by~\cite[Proposition 2.23]{Holkar-Hossain-2024-KMS-states}. Note that \(u_n\) can also be think as sections of the bundle \(p\colon \A|_{G_Z} \to G_Z\). 

\begin{lemma}\label{lem-ind-comp}
For \(f\in \Cc(G;\A)\), \(\inpro{u_n}{f*_lu_m} \to f|_{G(Z)}\) pointwise as \(m,n\to \infty\). Here, \(m,n\to \infty\) means one of the iterated limits \(m\to \infty \) and then \(n\to \infty\) or \(n\to \infty\) and then \(m\to \infty\). 
\end{lemma}
\begin{proof}
	For \(\gamma\in G(Z)\), we have 
	\begin{align*}
	\inpro{u_n}{f*_lu_m} (\gamma) &= \sum_{\alpha \in G_{r(\gamma)}} u^*_n(\alpha) f*_lu_m(\alpha\gamma) 
	 = u_n(r(\gamma)) f*_lu_m(\gamma) \\
	 &=  u_n(r(\gamma))  \sum_{\eta \in G^{r(\gamma)}} f(\eta) u_m(\eta^{-1}\gamma) 
	 =  u_n(r(\gamma)) f(\gamma) u_m(s(\gamma)).
	\end{align*}
	Since \((u_n(r(\gamma)))_{n\in\N}\) and  \((u_m(s(\gamma)))_{m\in\N}\) are approximation identity for \(\A_{r(\gamma)}\) and \(\A_{s(\gamma)}\) and \(\A_{\gamma}\) is an imprimitivity bimodule from \(\A_{r(\gamma)} \to \A_{s(\gamma)}\) for \(\gamma\in G(Z)\), taking iterated limit as \(m,n \to \infty\) in the last term above, we obtain
	\[
	 \lim_{m,n\to \infty} \inpro{u_n}{f*_lu_m}(\gamma) = \lim_{m,n\to \infty} u_n(r(\gamma)) f(\gamma) u_m(s(\gamma) = f(\gamma).
	\]
\end{proof}

\begin{lemma}\label{lem-unit-vector}
	Let \(\varphi\) be a state on \(\Cst(G(Z); \A|_{G(Z)})\) with GNS representation \((H, \pi, \xi)\). Then the sequence of vectors \((u_n\otimes \xi)_{n\in \N}\) in the Hilbert space \(X\otimes H\) is convergent and we denote the limit by \(a\). The limit \(a\) is a unit vector.
\end{lemma}
\begin{proof}
	We show the sequence \((u_n\otimes \xi)_{n\in \N}\) is Cauchy. Consider the norm in \(X\otimes H\)
	\begin{multline*}
		\norm{u_n\otimes \xi - u_m\otimes \xi}^2 = \inpro{u_n\otimes \xi - u_m\otimes \xi}{u_n\otimes \xi - u_m\otimes \xi}	\\
		= \inpro{u_n\otimes \xi}{u_n\otimes \xi} - \inpro{u_m\otimes \xi}{u_n\otimes \xi} -\inpro{u_n\otimes \xi}{u_m\otimes \xi} +\inpro{u_m\otimes \xi}{u_m\otimes \xi}\\
		= \inpro{\xi}{\pi(\inpro{u_n}{u_n})\xi} -\inpro{\xi}{\pi(\inpro{u_m}{u_n})\xi} -\inpro{\xi}{\pi(\inpro{u_n}{u_m})\xi} +\inpro{\xi}{\pi(\inpro{u_m}{u_m})\xi}\\
		= \inpro{\xi}{\pi(u^*_nu_n)\xi} -\inpro{\xi}{\pi(u^*_mu_n)\xi} -\inpro{\xi}{\pi(u^*_nu_m)\xi} +\inpro{\xi}{\pi(u^*_mu_m)\xi}\\
		\to \inpro{\xi}{\xi} -\inpro{\xi}{\xi} -\inpro{\xi}{\xi} +\inpro{\xi}{\xi} = 0 \quad \textup{as } m,n \to \infty.
		\end{multline*}
		Therefore, \((u_n\otimes \xi)_{n\in \N}\) is a Cauchy sequence. Let \(a=\lim_{n\to \infty} u_n\otimes \xi\) be the limit in \(X\otimes H\). The norm of the vector \(a\) given by 
		\begin{align*}
		\norm{a}^2 &= \norm{\lim_{n\to \infty} u_n\otimes\xi}^2 = \lim_{n\to \infty} \inpro{u_n\otimes \xi}{u_n\otimes \xi} = \lim_{n\to \infty} \inpro{\xi}{\pi(\inpro{u_n}{u_n})\xi}\\
		 &= \lim_{n\to \infty} \inpro{\xi}{\pi(u^*_nu_n)\xi} = \inpro{\xi}{\xi} = \norm{\xi}^2 =1.
		\end{align*}
\end{proof}

\begin{lemma}\label{lem-ind-rep-form}
Suppose \(\varphi\) is a state on \(\Cst(G(Z); \A|_{G(Z)})\) with the GNS representation \((H, \pi, \xi)\). Then, for \(f\in \Cc(G;\A)\), we have 
\begin{enumerate}
	\item \(f*_lu_n \to f|_{G_Z}\) in \(\Cc(G_Z;\A|_{G_Z})\) in the inductive limit topology;
	\item  \(\Ind (\pi) (f)(a) = f|_{G_Z}\otimes \xi\) in \(X\otimes H\);
	\item \(\inpro{a}{\Ind(\pi)(f)(a)} = \varphi (f|_{G(Z)})\).
\end{enumerate}
 \end{lemma}
 \begin{proof}
 	\noindent (1). For \(f\in \Cc(G;\A)\) and \(\gamma \in G_Z\), we have
 	\[
 	 f*_lu_n(\gamma)  =\sum_{h\in G^{r(\gamma)}} f(h) u_n(h^{-1}\gamma) = f(\gamma) u_n(s(\gamma)).
 	\]
 	The \(\supp(f*_lu_n) \subseteq \supp(f)\), which is a compact set and 
 	\[
 	 \norm{f*_lu_n-f|_{G_Z}}_{\infty} = \max_{\gamma \in \supp(f)} \norm{f(\gamma)u_n(s(\gamma))-f(\gamma)}.
 	\]
 	Since \(\A_{\gamma}\) is a Hilbert \(\A_{s(\gamma)}\)-module and \((u_n(s(\gamma)))_{n\in \N}\) is an approximation identity for \(\A_{s(\gamma)}\), the last term goes to zero as \(n\to \infty\). Therefore, \(f*_lu_n \to f|_{G_Z}\) in \(\Cc(G_Z;\A|_{G_Z})\) in the inductive limit topology.
 	
 	\noindent (2). For \(f\in \Cc(G;\A)\), we have 
 	\[ \Ind(\pi)(f)(a) = \lim_{n\to \infty} \Ind(\pi)(f) (u_n\otimes \xi) = \lim_{n\to \infty} (f*_lu_n)\otimes \xi = f|_{G_Z}\otimes \xi. 
    \] 
    The first equality follows from the continuity of \(\Ind (\pi)\) and the last one follows from part (1) of the current lemma.
    
    \noindent (3). Using the continuity of the induced representation \(\Ind (\pi)\), we have
    \begin{align*}
    \inpro{a}{\Ind(\pi)(f)(a)} &= \lim_{m,n\to \infty} \inpro{u_n\otimes \xi}{\Ind (\pi)(f) (u_m\otimes \xi)} \\
    &=
     \lim_{m,n\to \infty} \inpro{u_n\otimes\xi}{(f*_lu_m)\otimes\xi}
     = \inpro{\xi}{\pi(\inpro{u_n}{f*_lu_m})\xi}
   \end{align*}
   Using Lemma~\ref{lem-ind-comp} the last term can be written as
   \[
    \inpro{\xi}{\pi(f|_{G(Z)})\xi} = \varphi(f|_{G(Z)}).
   \]
 \end{proof}
 
 \begin{proposition}\label{prop-cyclic-vector}
 If \(\varphi\) is a state on \(\Cst(G(Z); \A|_{G(Z)})\) with the GNS representation \((H, \pi, \xi)\). Then \(a\) is a cyclic vector for the induced representation \(\Ind(\pi) \colon \Cst(G;\A) \to \Bound(X\otimes H)\).
 \end{proposition}
 \begin{proof}
To prove \(a\) is a cyclic vector for \(\Ind(\pi)\), we need to show that 
\[
 M=\textup{span}\{\Ind(\pi)(f)(a) : f\in \Cc(G;\A)\}
\] 	
is dense in \(X\otimes H\). Let \(N\) be the vector subspace of \(X\otimes H\) span by the vectors \(\eta\otimes \pi(h)\xi\) where \(\eta\in \Cc(G_Z;\A|_{G_Z})\) and \(h\in \Cc(G(Z); \A|_{G(Z)})\). Since \(\xi\) is a cyclic vector for the representation \(\pi\) of \(\Cst(G(Z); \A|_{G(Z)})\) on the Hilbert space \(H\), we have \(N\) is dense in \(X\otimes H\). We now show that \(N\subseteq M\), which will prove the proposition. For this, we consider \(\eta\otimes \pi(h)\xi \in N\). Since the tensor product \(\otimes\) is balanced over \(\Cst(G(Z); \A|_{G(Z)})\), we have \(\eta\otimes \pi(h)\xi = \eta *_rh\otimes \xi\) where \(\eta*_rh\in \Cc(G_Z;\A|_{G_Z})\). 

We claim that there exists \(f\in \Cc(G;\A)\) such that \(f|_{G_Z} = \eta*_rh\). Then using this claim and Lemma~\ref{lem-ind-rep-form}(2), we have 
\[
(\eta*_rh)\otimes \xi = f|_{G_Z}\otimes \xi = \Ind(\pi)(f)(a) \in M.
\] 
Therefore, \(N\subseteq M\). Thus to complete this proof we only need to establishes the claim. But Since \(G_Z\) is a closed set of \(G\), this claim is true from vector-valued Tietze extension theorem (see~\cite[Proposition A.5]{Muhly-Williams2008Equivalence-and-disinte-thm-Fell-bundle}). This concludes the proof.
 \end{proof}
 
 \begin{theorem}\label{thm-ind-state}
 Let \(p\colon \A \to G\) be Fell bundle over a locally compact Hausdorff second countable \'etale groupoid \(G\) with a \(1\)\nb-cocycle \(c\colon G\to \R\). If \(\varphi\) is a state on \(\Cst(G(Z); \A|_{G(Z)})\) with the GNS representation \((H,\pi, \xi)\), then there exist a state \(\psi_{\varphi} \) on \(\Cst(G;\A)\) such that 
 \[
  \psi_{\varphi}(f) =\varphi(f|_{G(Z)})
 \]
 for \(f\in \Cc(G;\A)\).
 \end{theorem}
 \begin{proof}
 	Since \(\Ind(\pi)\) is a representation of \(\Cst(G;\A)\) on \(X\otimes H\), and \(a\) is a cyclic vector by Lemma~\ref{lem-unit-vector} and Proposition~\ref{prop-cyclic-vector}, the triple \((X\otimes H, \Ind(\pi), a)\) gives us a state \(\psi_{\varphi}\) on \(\Cst(G;\A)\) given by 
 	\[
 	 \psi_{\varphi}(f) = \inpro{a}{\Ind(\pi)(f)a} = \varphi(f|_{G(Z)})
 	\]
 	for \(f\in \Cc(G;\A)\). The second equality follows from Lemma~\ref{lem-ind-rep-form}(3).
 \end{proof}

\section{Ground states on Fell bundle \(\Cst\)-algebras}\label{sec-main-ground-st}
Let \((A,\sigma)\) be a real dynamical system. Recall the definition of ground states from Section~\ref{subsec-KMS-ground-st}. 
Bratteli--Robinson~\cite[Proposition 5.3.19]{Bratteli-Robinson1981Oper-alg-Quan-sta-mech-part-2} give the following equivalent characterization of a ground state, which we will use to prove our main result. A state \(\varphi\) on \(A\) is a \(\sigma\)\nb-ground state if and only if for every compactly supported smooth function \(F\) on \(\R\) supported in \((-\infty, 0)\) we have 
\begin{equation}\label{eq-ground-st-equiv-cond}
\varphi(\sigma_{\check{F}}(a)^*\sigma_{\check{F}}(a)) = 0 \quad \textup{for all } a\in A,
\end{equation}
where \(\check{F}\) is the inverse Fourier transform of \(F\), i.e., 
\[
 \check{F}(x) = \frac{1}{2\pi}\int_{\R} F(y)\e^{-ixy} \dd y
\]
and 
\[
 \sigma_{\check{F}}(a) = \int_{R} \check{F}(t)\sigma_t(a)\dd t.
\]
For a Fell bundle \(\Cst\)\nb-algebra \(A=\Cst(G;\A)\), the last equation became 
\begin{equation}\label{eq-dyna-for-Fell-Four}
\sigma^c_{\check{F}}(f)(\gamma) = \int_{\R} \check{F}(t) \sigma^c_t(f) (\gamma) \dd (t) =  \int_{\R} \check{F}(t) \e^{itc(\gamma)} f(\gamma)\dd (t) = F(c(\gamma)) f(\gamma)
\end{equation}
for \(\gamma \in G\) and \(f\in \Cc(G;\A)\).

\begin{proposition}\label{prop-one-dir-ground}
	Let \(p\colon \A \to G\) be a Fell bundle over a locally compact Hausdorff second countable \'etale groupoid \(G\) and let \(c\colon G\to \R\) be a \(1\)\nb-cocycle. Let \(\sigma^c\) be the real dynamics on \(\Cst(G;\A)\) associated with \(c\) as in Equation~\eqref{equ:real-dyna}. If \(Z\neq \emptyset\) and the boundary groupoid \(G(Z)\) is \'etale, then every state \(\varphi\) on \(\Cst(G(Z);\A|_{G(Z)})\) gives a \(\sigma^c\)-ground state \(\psi_{\varphi}\) on \(\Cst(G;\A)\) given by  
\begin{equation*}
	\psi_{\varphi}(f) = \varphi(f|_{G(Z)})
\end{equation*}
for \(f\in \Cc(G;\A)\).
\end{proposition}
\begin{proof}
 Let \((H, \pi, \xi)\) be the GNS representation of \(\varphi\). By Corollary~\ref{cor-ind-rep}, \(\Ind(\pi)\colon \Cst(G;\A) \to \Bound(X\otimes H)\) is a representation, where \(X\) is the \(\Cst\)\nb-correspondence from \(\Cst(G;\A) \to \Cst(G(Z); \A|_{G(Z)})\) given by Theorem~\ref{thm-C-st-corr}. Recall from Proposition~\ref{prop-cyclic-vector} that the induced representation \(\Ind(\pi)\) has a cyclic vector \(a = \lim_{n\to \infty} u_n\otimes \xi\). Moreover, Theorem~\ref{thm-ind-state} gives the existence of a state \(\psi_{\varphi}\) on \(\Cst(G;\A)\) such that 
\[
\psi_{\varphi}(f) = \varphi(f|_{G(Z)}).
\]
We need to show that \(\psi_{\varphi}\) is a \(\sigma^c\)-ground state on \(\Cst(G;\A)\).

For \(t\in \R\), define an operator \(U_t\) on \(X\) by 
\[
U_t(g)(\gamma) = \e^{itc(\gamma)}g(\gamma)
\]
for \(g\in \Cc(G_{Z}; \A|_{G_Z})\). Since the \(1\)-cocycle \(c\) vanishes on \(G(Z)\), the operator \(U_t\) compatible with the inner product given by Equation~\eqref{eq-equivalence-3} and right action \(*_r\) given by Equation~\eqref{eq-equivalence-2}. Hence, we obtain a one-parameter group of unitary operators \((U_t)_{t\in \R}\) on the Hilbert \(\Cst(G(Z); \A|_{G(Z)})\)-module \(X\). Moreover, the unitary \(U_t\) implemented the dynamics \(\sigma^c\), i.e.,
\begin{align*}
	U_t(f*_l\zeta) (\gamma) &= \e^{itc(\gamma)}f*_l\zeta(\gamma) 
	= \e^{itc(\gamma)} \sum_{\eta\in G^{r{\gamma}}} f(\eta)\zeta(\eta^{-1}\gamma)\\
	&= \sum_{\eta\in G^{r(\gamma)}} \e^{itc(\eta)} \e^{itc(\eta^{-1}\gamma)} f(\eta)\zeta(\eta^{-1}\gamma) \\
	&=\sum_{\eta\in G^{r(\gamma)}} \sigma^c_t(f)(\eta)U_t(\zeta)(\gamma)
	= \sigma^c_t(f)*_lU_t(\zeta)(\gamma)
\end{align*}
for \(f\in \Cc(G;\A), \zeta\in \Cc(G_Z;\A|_{G_Z})\) and \(\gamma \in G_Z\). Define the unitary \(W_t\defeq U_t\otimes 1\) on the Hilbert space \(X\otimes H\). Then, we have 
\begin{align}\label{eq-Ind-W-imple}
	W_t\Ind(\pi)(f)(\zeta\otimes h) &= W_t((f*_l\zeta)\otimes h) = (\sigma^c_t(f)*_lU_t(\zeta)\otimes) h \nonumber\\
	&= \Ind(\pi) (\sigma^c_t(f))(U_t(\zeta)\otimes h) = \Ind(\pi) (\sigma^c_t(f)) W_t(\zeta\otimes h)
\end{align}
for \(f\in \Cc(G;\A), \zeta \in \Cc(G_Z;\A|_{G_Z})\) and \(h\in H\).
Thus, the dynamics \(\sigma^c\) implemented by \(W_t\). Since \(c=0\) on \(G(Z)\), we have \(U_t(u_n) = u_n\). Thus, we have 
\begin{equation}\label{eq-a-inv-under-W}
	W_ta= W_t\lim_{n\to \infty} u_n\otimes \xi = \lim_{n\to \infty} W_t(u_n\otimes \xi) = \lim_{n\to \infty} u_n\otimes \xi = a.
\end{equation}
 Let  \(F\) be a smooth function supported on \((-\infty, 0)\). Using Lemma~\ref{lem-ind-rep-form}(3), we have 
\begin{align}\label{eq-GNS-psi}
	\psi_{\varphi} (\sigma^c_{\check{F}}(f)^*\sigma^c_{\check{F}}(f)) &= \inpro{a}{\Ind(\pi)(\sigma^c_{\check{F}}(f)^*\sigma^c_{\check{F}}(f))a} \nonumber\\
	&= \inpro{\Ind(\pi)(\sigma^c_{\check{F}}(f))a}{\Ind(\pi)(\sigma^c_{\check{F}}(f))a} =\norm{\Ind(\pi)(\sigma^c_{\check{F}}(f))a}^2.
\end{align}
Equation~\eqref{eq-Ind-W-imple} and the definition of \(\sigma^c_{\check{F}}\) give us 
\begin{align}\label{eq-compt-of-Ind}
	\Ind(\pi) (\sigma^c_{\check{F}}(f)) (a) &= \int_{\R} \check{F}(t) W_t\Ind(\pi)(f)W_t^*(a)\dd t \nonumber\\
	&= \int_{\R} \check{F}(t) W_t\Ind(\pi) (f) (a) \dd t \nonumber\\
	&= W_{\check{F}} \Ind(\pi)(f) (a) 
\end{align}
for \(f\in \Cst(G;\A)\).
Here, the second equality follows from Equation~\eqref{eq-a-inv-under-W} and third one follows from
\[
W_{\check{F}} = \int_{\R} \check{F}(t)W_t\dd t.
\]
 Lemma~\ref{lem-ind-rep-form}(2) give us 
\[
\Ind(\pi)(f)(a) = f|_{G_Z} \otimes \xi
\]
for \(f\in \Cc(G;\A)\). Applying the unitary \(W_t\) we obtain
\begin{equation}\label{eq-formula-for-Ind-W}
	W_t\Ind(\pi)(f)(a) = U_t\otimes 1(f|_{G_Z} \otimes \xi) = U_t(f|_{G_Z}) \otimes \xi = \e^{itc(\cdot)} f|_{G_Z}(\cdot) \otimes \xi.
\end{equation} 
The last equality follows from the definition of \(U_t\). Now using the definition of \(W_{\check{F}}\) we obtain 
\begin{align*}
W_{\check{F}}\Ind(\pi)(f)(a) &= \int_{\R} \check{F}(t)W_t\Ind(\pi)(f)(a) \dd t \\
&= \int_{\R} \check{F} (t) \e^{itc(\cdot)} f|_{G_Z}(\cdot) \otimes \xi \dd t \\
&= F\circ c(\cdot) f|_{G_Z} (\cdot) \otimes \xi = (F\circ c)|_{G_Z}f|_{G_Z} \otimes \xi.
\end{align*}
The second equality follows from Equation~\eqref{eq-formula-for-Ind-W} and third one from the Fourier inversion formula. As the cocycle \(c\geq 0\) on \(G_Z\), we have \( W_{\check{F}}\Ind(\pi)(f)(a) = 0\) for all \(f\in \Cc(G;\A)\). By Equation~\eqref{eq-compt-of-Ind}, \(
\Ind(\pi)(\sigma_{\check{F}}^c(f))a=0\) for \(f\in \Cc(G;\A)\). 
Together with Equation~\eqref{eq-GNS-psi}, this gives
\[
\psi_{\varphi}\bigl(
\sigma_{\check{F}}^c(f)^*
\sigma_{\check{F}}^c(f)\bigr)=0.
\]
By the characterization of ground states in
Equation~\eqref{eq-ground-st-equiv-cond}, \(\psi_{\varphi}\) is a
\(\sigma^c\)-ground state.
\end{proof}

\begin{theorem}\label{thm-main-ground-st}
	Let \(p\colon \A \to G\) be a Fell bundle over a locally compact Hausdorff second countable \'etale groupoid \(G\) and let \(c\colon G\to \R\) be a \(1\)\nb-cocycle. Let \(\sigma^c\) be the real dynamics on \(\Cst(G;\A)\) associated with \(c\) as in Equation~\eqref{equ:real-dyna}. If \(Z\neq \emptyset\) and the boundary groupoid \(G(Z)\) is \'etale, then there is an affine homeomorphism between the state space of \(\Cst(G(Z);\A|_{G(Z)})\) and the \(\sigma^c\)\nb-ground state space of \(\Cst(G;\A)\). This assignment sends a state \(\varphi\) on \(\Cst(G(Z); \A|_{G(Z)})\) to a unique \(\sigma^c\)-ground state \(\psi_{\varphi}\) on \(\Cst(G;\A)\) given by 
	\begin{equation}\label{eq-relation-state-ground}
	\psi_{\varphi}(f) = \varphi(f|_{G(Z)})
	\end{equation}
	for \(f\in \Cc(G;\A)\). Moreover, if \(Z=\emptyset\), then \(\Cst(G;\A)\) has no \(\sigma^c\)\nb-ground states.
\end{theorem}
\begin{proof}
	Let \(\psi\) be a \(\sigma^c\)-ground state on \(\Cst(G;\A)\). Suppose \((H, L, \xi )\) be the GNS representation of \(\psi\). Applying Muhly--Williams disintegration theorem~\cite{Muhly-Williams2008Equivalence-and-disinte-thm-Fell-bundle} to \(L\), we obtain a strict representation \((\mu, \Hilm, \hat{\pi})\) where \(\mu\) is a quasi-invariant measure on \(\base\), \(\Hilm = \{\Hilm_x\}_{x\in \base}\) is a \(\mu\)-measurable field of Hilbert spaces and \(\hat{\pi}\) is a \(*\)-functor. Moreover, \(L\) is equivalent to the integrated form of \((\mu, \Hilm, \hat{\pi})\) (see Equation~\eqref{equ-int-dis-rep}). We identify \(H\) with the direct integral \(\int_{\base}^{\oplus} \Hilm_x \dd \mu(x)\) and \(\xi\) can be identified with the vector field \((\xi(x))_{x\in \base}\). Then for \(f\in \Cc(G;\A)\), we have
	\begin{equation}\label{eq-dis-int-formula}
    L(f)\xi(x) = \sum_{\gamma\in G^x} \pi(f(\gamma)) \xi(s(\gamma)) \Delta^{-\frac{1}{2}}_{\mu} (\gamma)
	\end{equation}
	where \(\Delta_{\mu}\) is the modular function of the quasi-invariant measure \(\mu\).
	
	We first prove that \(\xi(x) = 0\) for \(\mu\)-a.e. \(x\in \base \setminus Z\). For this, let \(y\in \base\setminus Z\). By definition of the boundary set \(Z\), there exists \(\gamma^{\prime} \in G_y\) such that \(c(\gamma^{\prime}) <0\). Since \(G\) is \'etale, choose an open bisection \(U\) containing \(\gamma^{\prime}\) such that \(r|_{U} \colon U\to r(U)\) and \(s|_{U} \colon U\to s(U)\) are homeomorphisms and \(c(U)\subseteq (c(y)-\delta , c(y)+\delta)\) for some \(0<\delta <-c(\gamma^{\prime})\). Fix a compactly supported smooth function \(F\) on \(\R\) with \(\supp(F) \subset (-\infty, 0)\) and \(F \equiv 1\) on \([c(\gamma^{\prime})-\delta, c(\gamma^{\prime})+\delta]\). For \(f\in \Cc(G;\A)\) with \(\supp(f)\subseteq U\), the Equation~\eqref{eq-dis-int-formula} becomes
	\[
	 L(f)\xi(x) =  \begin{cases}
	 	\pi(f(r^{-1}(x)))\xi(s(r^{-1}(x))) \Delta_{\mu}^{-\frac{1}{2}}(r^{-1}(x)) & \textup{ if } x\in r(U),\\
	 	0 & \textup{ if } x\notin r(U).
	 \end{cases} 
	\]
	Equation~\eqref{eq-dyna-for-Fell-Four} gives us \(\sigma^c_{\check{F}}(f)(\gamma) =  F(c(\gamma))f(\gamma) = f(\gamma)\), thus using the last equation we have  
	\begin{align*}
	\norm{L(\sigma_{\check{F}}(f))\xi}^2 &= \int_{r(U)} \norm{\pi(f(r^{-1}(x)))}^2\norm{\xi(s(r^{-1}(x)))}^2 \Delta^{-1}_{\mu} (r^{-1}(x)) \dd \mu(x)\\
	 &=\int_{s(U)} \norm{\pi(f(s^{-1}(x)))}^2\norm{\xi(x)}^2 \dd \mu(x).
	\end{align*}
	The last equality follows from the fact that \(\mu\) is a quasi-invariant measure with Radon--Nikodym derivative \(\Delta_{\mu}\). 
	
	Since \(\psi\) is a \(\sigma^c\)-ground state, Equation~\eqref{eq-ground-st-equiv-cond} gives us
	\begin{align*}
		0= \psi((\sigma^c_{\check{F}})^*(f)\sigma^c_{\check{F}}(f)) &= \inpro{\xi}{L(\sigma^c_{\check{F}}(f)^*\sigma^c_{\check{F}}(f))\xi} = \inpro{L(\sigma^c_{\check{F}}(f))\xi}{L(\sigma^c_{\check{F}}(f))\xi}\\
		&= \norm{L(\sigma^c_{\check{F}}(f))\xi}^2 = \int_{s(U)} \norm{\pi(f(s^{-1}(x)))}^2\norm{\xi(x)}^2 \dd \mu(x).
		\end{align*}
		Since \(f\) was arbitrary section supported on \(U\), it follows that \(\xi(x) =0 \) for \(\mu\)-a.e. \(x\) in the neighborhood \(s(U)\) of \(y\). Since \(y\in \base \setminus Z\) was arbitrary, we have \(\xi(x) = 0\) for \(\mu\)-a.e. \(x\in \base\setminus Z\). 
		
		Since \(\xi\) is a unit cyclic vector, \(\mu(Z)>0\)  and \(Z\neq \emptyset\). And if \(Z=\emptyset\), there can not exists a ground state, thus we assume \(Z\neq \emptyset\).
		
		Let \(\mu|_{Z}\) be the restriction of the measure \(\mu\) on \(Z\). We consider the bundle of Hilbert spaces \(\Hilm|_{Z} \defeq \{\Hilm_x\}_{x\in Z}\) and identify \(H_Z\defeq \int_{Z}^{\oplus} \Hilm_x \dd\mu(x)\) with a subspace of \(\int_{\base}^{\oplus}  \Hilm_x \dd\mu(x)\). Then \((\Hilm|_{Z}, \hat{\pi}|_Z, \mu|_Z)\) is a strict representation of the Fell bundle \(p\colon\ \A|_{G(Z)} \to G(Z)\), where \(\hat{\pi}|_Z(a) = (r(p(a)), \pi(a), s(p(a)))\) for \(a\in \A|_{G(Z)}\). Since \(\xi(x) = 0\) for \(\mu\)-a.e. \(x\in \base \setminus Z\), we have \(\xi\in \Hilm|_Z\). Thus, the integrated form of the strict representation \(\Hilm|_{Z} \defeq \{\Hilm_x\}_{x\in Z}\) gives a representation of \(\Cst(G(Z); \A|_{G(Z)})\) on the Hilbert space \(H_Z\). This representation together with the cyclic vector \(\xi\) determines the require state \(\varphi\) on \(\Cst(G(Z);\A|_{G(Z)})\). This state is given by
		\begin{equation}\label{eq-dis-int-phi}
		\varphi (g) = \int_{Z} \sum_{\gamma\in G^x} \inpro{\xi(x)}{\pi(g(\gamma))\xi(s(\gamma))} \Delta^{-\frac{1}{2}}_{\mu} (\gamma) \dd \mu(x).
	   \end{equation}
	   for \(g\in \Cc(G(Z); \A|_{G(Z)})\).
    On the other hand for \(f\in\Cc(G;\A)\), we have 
    \begin{multline*}
     \psi(f) = \inpro{\xi}{L(f)\xi} = \int_{\base} \inpro{\xi(x)}{L(f)\xi(x)} \dd\mu(x) = \int_{Z} \inpro{\xi(x)}{L(f)\xi(x)} \dd\mu(x)\\
      = \int_{Z} \sum_{\gamma\in G^x} \inpro{\xi(x)}{\pi(f(\gamma))\xi(s(\gamma))} \Delta^{-\frac{1}{2}}_{\mu} (\gamma)\dd \mu(x)
     = \varphi (f|_{G(Z)}).
      \end{multline*}
   The first equality of the second line follows from Equation~\eqref{eq-dis-int-formula} and the second one follows from Equation~\eqref{eq-dis-int-phi}.
   The state \(\varphi\) is unique because each \(f\in \Cc(G(Z); \A|_{G(Z)})\) can be extended to a compactly supported section of the bundle \(p\colon \A\to G\) (by vector-valued Tietze extension theorem~\cite[Proposition A.5]{Muhly-Williams2008Equivalence-and-disinte-thm-Fell-bundle}), thus \(\varphi\) is the only state satisfying Equation~\eqref{eq-relation-state-ground}.

 Conversely, Proposition~\ref{prop-one-dir-ground} shows that every
 state \(\varphi\) on \(\Cst(G(Z);\A|_{G(Z)})\) gives a
 \(\sigma^c\)\nb-ground state \(\psi_{\varphi}\) on \(\Cst(G;\A)\)
 satisfying
 \[
 \psi_{\varphi}(f)=\varphi(f|_{G(Z)})
 \]
 for \(f\in \Cc(G;\A)\). Also, every \(\sigma^c\)-ground
 state arises in this way and that the corresponding state
 \(\varphi\) is unique. Thus, this correspondence is bijective. The correspondence  is affine, as
 \[
 \psi_{\lambda\varphi_1+(1-\lambda)\varphi_2}
 =
 \lambda\psi_{\varphi_1}
 +
 (1-\lambda)\psi_{\varphi_2}
 \]
 for \(0\leq \lambda \leq 1\).
 Moreover, the correspondence is weak\(^*\)-continuous by the Equation~\eqref{eq-relation-state-ground}.
 Since the state spaces are weak\(^*\)-compact and Hausdorff, the
 bijection is a homeomorphism. This completes the proof.
 \end{proof}
 
 \begin{remark}
 	Proposition~4.11 of~\cite{Holkar-Hossain-2024-KMS-states} provides a natural class of states on \(\Cst(G(Z); \A|_{G(Z)})\). Indeed, any pair \((\mu,\{\varphi_x\}_{x\in Z})\), consisting of a probability measure \(\mu\) on \(Z\) and a \(\mu\)\nb-measurable field of states \(\{\varphi_x\}_{x\in Z}\) on \(\Cst(G_x^x;\A|_{G_x^x})\), gives rise to an integrated state on \(\Cst(G(Z);\A|_{G(Z)})\). Consequently, by Theorem~\ref{thm-main-ground-st}, each such integrated state determines a \(\sigma^c\)\nb-ground state on \(\Cst(G;\A)\).
 \end{remark}

\noindent Recall the definition of KMS\(_{\infty}\) states from Section~\ref{subsec-KMS-ground-st}. 
\begin{corollary}\label{coro-infi-stat}
	Let \(p\colon \A \to G\) be a Fell bundle over a locally compact Hausdorff second countable \'etale groupoid \(G\) and let \(c\colon G\to \R\) be a \(1\)\nb-cocycle. Let \(\sigma^c\) be the real dynamics on \(\Cst(G;\A)\) associated with \(c\) as in Equation~\eqref{equ:real-dyna}. Assume \(Z\neq \emptyset\) and \(c^{-1}(0)\) is \'etale. If the \(\sigma^c\)-ground state \(\psi_{\varphi}\) on \(\Cst(G;\A)\) is a KMS\(_{\infty}\) state, then \(\varphi\) is tracial.
\end{corollary}
\begin{proof}
Let \(\varphi\) be a state on \(\Cst(G(Z); \A|_{G(Z)})\) and let \(\psi_{\varphi}\) be the associated \(\sigma^c\)-ground state given by Theorem~\ref{thm-main-ground-st}. Assume \(\psi_{\varphi}\) is a KMS\(_{\infty}\) state. The definition of the dynamics \(\sigma^c\) and the KMS condition ensures that the restriction of \(\psi_{\varphi}\) on \(\Cc(c^{-1}(0); \A|_{c^{-1}(0)}) \subseteq \Cst(G;\A)\) is tracial. Since the boundary groupoid \(G(Z)\) is the reduction of \(c^{-1}(0)\) on the boundary set \(Z\), which is also \(c^{-1}(0)\)\nb-invariant, the restriction map \(f\mapsto f|_{G(Z)}\) on \(\Cc(c^{-1}(0))\) can be extended to homomorphism \(\rho\colon \Cst(c^{-1}(0); \A|_{c^{-1}(0)}) \to \Cst(G(Z); \A|_{G(Z)})\). By vector-valued Tietze extension theorem (see~\cite[Proposition A.5]{Muhly-Williams2008Equivalence-and-disinte-thm-Fell-bundle}) gives us the surjectivity of \(\rho\). Since the restriction of \(\psi_{\varphi}\) to \(\Cst(c^{-1}(0); \A|_{c^{-1}(0)})\) is same as \(\varphi\circ \rho\), the state \(\varphi\) is tracial.
\end{proof}

\begin{remark}\label{rmk-Infinity}
	Corollary~\ref{coro-infi-stat} shows that every KMS\(_\infty\) state arising from
	the boundary set \(Z\) restricts to a tracial state on
	\(\Cst(G(Z);\A|_{G(Z)})\). In particular, whenever
	\(\Cst(G(Z);\A|_{G(Z)})\) admits non-tracial states, the
	corresponding ground states are not KMS\(_\infty\) states.
\end{remark}

\section{Crystallization of Fell bundle \(\Cst\)-algebras}\label{sec-crystal}
In this section, we first recall the crystallization of a \(\Cst\)\nb-dynamical system \((A, \sigma)\) and its relation to ground states, following~\cite{Laca-Neshveyev-Yamashita-Crystallization}. We then identify the crystal associated to the \(\Cst\)\nb-algebra of a Fell bundle.

Let \((A,\sigma)\) be a real \(\Cst\)-dynamical system. For \(\lambda \in \R\), define
\[
A_{\lambda} \defeq \{a\in A: \sigma_t(a) =\e^{it\lambda} a \textup{ for all } t\in\R\}.
\]
The dynamics \(\sigma\) is called \emph{almost periodic} if the \(^*\)-subalgebra generated by \(A_{\lambda}\) for \(\lambda \in \R\) is dense in \(A\). We assume the dynamics \(\sigma\) is almost periodic. Let \(I_{\lambda}\) be the ideal \(\overline{A_{\lambda}A^*_{\lambda}}\) of the fixed-point algebra \(A_0\) and set 
\begin{equation}\label{eq-ideal-crystal}
I = \overline{\sum_{\lambda>0} I_{\lambda}}.
\end{equation}
The \(\Cst\)-algebra  \(A_c\defeq A_0/I\) is called the \emph{crystal} of \((A,\sigma)\) (see~\cite[Defintion 1.1]{Laca-Neshveyev-Yamashita-Crystallization}).

Let \(\Gamma \subset \mathbb{R}\) be the additive subgroup generated by those \(\lambda \in \mathbb{R}\) for which \(A_\lambda \neq 0\). We consider \(\Gamma\) as a discrete group. The inclusion \(\Gamma \hookrightarrow \R\) induces, by duality, a homomorphism from \(\R \to \widehat{\Gamma}\). The image of \(\R\) is dense in \(\hat{\Gamma}\). Thus, the map \(\R \to A\) by  \(t\mapsto \sigma_t(a)\) extends to an action of \(\widehat{\Gamma}\) on \(A\). We continue to denote this action by~\(\sigma\). This action gives rise to a conditional expectation \(E \colon A \to A_0\) defined by
\[
E(a)=\int_{\widehat{\Gamma}} \sigma_{\chi}(a) \dd\chi,
\]
Moreover, \(
E(A_\lambda)=0\) for \(\lambda\neq 0\).

\begin{proposition}[{\cite[Proposition 1.2 ]{Laca-Neshveyev-Yamashita-Crystallization}}]
Let \((A, \sigma)\) be an almost periodic dynamics. Then there is a completely positive contraction \(\theta \colon A\to A_c\) given by \(\theta (a) = E(a)+I\). If \(A_c\neq 0\), then \(\varphi\mapsto \varphi\circ\theta\) gives a bijective correspondence between the state space of \(A_c\) and the \(\sigma\)-ground state space of \(A\). If \(A_c=0\), then \(A\) has no \(\sigma\)-ground states.
\end{proposition}  

To characterize the crystal of a Fell bundle \(\Cst\)-algebra \(\Cst(G;\A)\), we need the following result of Ionescu--Williams~\cite[Theorem 3.7]{Ionescu-Williams2012Remarks-on-ideal-stru-Fell-bundle-cst-alg}.

\begin{proposition}[Ionescu--Williams]\label{prpo-exact-seq-Will}
Let \(p\colon \A\to G\) be a Fell bundle over a locally compact Hausdorff second countable groupoid \(G\) and let \(Y\subseteq \base\) be an open invariant subset. Then there exists a short exact sequence
\[
0 \longrightarrow \Cst(G(Y); \A|_{G(Y)}) \xlongrightarrow{\iota} \Cst(G;\A) \xlongrightarrow{r} \Cst(G(\base\setminus Y); \A|_{G(\base\setminus Y)})\longrightarrow 0.
\]
 \end{proposition}
 
 \begin{remark}
 	The short exact sequence in Proposition~\ref{prpo-exact-seq-Will} is a special
 	case of the ideal structure result of Ionescu--Williams~\cite[Theorem 3.7]{Ionescu-Williams2012Remarks-on-ideal-stru-Fell-bundle-cst-alg}.
 	Indeed, let \(Y\) be an open invariant subset of \(\base\) and
 	consider the \(G\)\nb-invariant ideal \(	I=\Contz(Y;\A|_Y)\)
 	of \(\Contz(\base;\A|_{\base})\). The Fell bundle associated to \(I\)
 	is naturally identified with the restriction \(\A|_{G(Y)}\), while the
 	quotient Fell bundle is naturally identified with
 	\(\A|_{G(\base \setminus Y)}\). Thus, the short exact sequence in
 	Proposition~\ref{prpo-exact-seq-Will} follows directly from the main result of~\cite{Ionescu-Williams2012Remarks-on-ideal-stru-Fell-bundle-cst-alg}.
 \end{remark}
 
 
 \noindent For the next result, we assume that \(c\) is a locally constant \(1\)\nb-cocycle, i.e., \(c\colon G\to \R\) is a continuous groupoid homomorphism when \(\R\) equipped with the discrete topology.
 \begin{theorem}
 Let \(p\colon \A\to G\) be a Fell bundle over a locally compact Hausdorff second countable \'etale groupoid \(G\) and let \(c\colon G\to \R\) be a locally constant \(1\)-cocycle with the associated boundary set \(Z\). Let \(\sigma^c\) be the associated dynamics given by Equation~\eqref{equ:real-dyna}. Then the crystal \(A_c\) of \((\Cst(G;\A), \sigma^c)\) is canonically isomorphic to \(\Cst(G(Z)); \A|_{G(Z)})\). In particular \(A_c\neq 0\) if and only if \(Z\neq \emptyset\).  
 \end{theorem}
\begin{proof}
Set \(K=c^{-1}(0)\).
Since \(c\) is locally constant, \(K\) is a clopen subgroupoid of \(G\). Theorem~3.23 of~\cite{Hossain-Inclusion-Fell-bund} 	ensures that the closure of \(\Cc(K; \A|_K)\) in \(\Cst(G;\A)\) can be identified with \(\Cst(K;\A|_{K})\). It follows from the definition that \(A_0\) is the closure of \(\Cc(K;\A|_{K})\) in \(\Cst(G;\A)\), thus \(A_0 =\Cst(K;\A|_{K})\). Since the boundary set \(Z\) is invariant under the action of the subgroupoid \(K\), we have \(G(Z)= K(Z)\) (by~\cite[Proposition I.3.16]{Renault1980A-gpd-appr-to-cst-alg}). 
Consider the \(^*\)-homomorphism \(\rho\colon \Cst(K;\A|_{K}) \to \Cst(G(Z); \A|_{G(Z)})\) by
\(\rho(f) = f|_{G(Z)}\) for \(f\in \Cc(K;\A)\). Since \(Z\) is invariant, applying Proposition~\ref{prpo-exact-seq-Will} to the open invariant set \(\base\setminus Z\), we obtain the short exact sequence
\[
0\longrightarrow
\Cst(K(K^{(0)}\setminus Z);\A|_{K(K^{(0)}\setminus Z)}) \longrightarrow \Cst(K;\A|_K) \xrightarrow{\rho}
\Cst(G(Z);\A|_{G(Z)}) \longrightarrow 0.
\]
Since \(\rho\) is surjective, we have 
\[
 \Cst(G(Z); \A|_{G(Z)}) \cong \frac{\Cst(K; \A|_{K})}{\textup{ker}(\rho)} \cong \frac{A_0}{\textup{ker}(\rho)}.
\]  
Thus, we need to prove that \(\ker(\rho) =I\), where \(I\) is given by Equation~\eqref{eq-ideal-crystal}. Let \(Y\defeq  K\setminus K(Z) =K\setminus G(Z)\). By the above short exact sequence we can see that the kernel \(\ker(\rho)\) contains \(\Cc(Y;\A|_{Y})\) as a dense subspace. To complete the proof we need to prove that the algebraic ideal 
\[
 J = \textup{span}\{s_1*s^*_2: s_1,s_2\in \Cc(c^{-1}(\lambda), \A|_{c^{-1}(\lambda)}) \textup{ for } \lambda >0\} 
\]
of \(\Cst(K;\A|_{K})\) is dense in \(\Contz(Y;\A|_Y)\). We use~\cite[Lemma A. 4]{Muhly-Williams2008Equivalence-and-disinte-thm-Fell-bundle} to prove this density part. Let \(g\in \Contz(Y)\). Since \(G\) is an \'etale groupoid, using a partition of unity argument we assume that \(g\) supported in a bisection \(U\). Thus, \(h=g\circ (r|_U)^{-1} \in \Contz(r(U))\). Now for \(s_1*s_2^*\in J\) supported in a bisection, we have 
\[
 g\cdot (s_1*s^*_2) (\gamma) = g(\gamma) (s_1*s^*_2)(\gamma) = h(r(\gamma)) (s_1*s^*_2) (\gamma) = (hs_1)*s^*_2(\gamma),
\]
where \(hs_1(\gamma)  =h(r(\gamma))s_1(\gamma)\) for \(\gamma \in U\) and \(h\in  \Contz(r(U))\). Therefore, \( g\cdot (s_1*s^*_2) \in J\). 

Now we establishes fibrewise density. Fix a point \(\gamma \in Y = c^{-1}(0)\setminus G(Z)\). From the definition of \(G(Z)\), it follows that \(r(\gamma) \notin Z\) and there exists \(\eta \in G^{r(\gamma)}\) with \(c(\eta) =\lambda >0\). 
Choose small bisections \(U\ni \gamma\) and \(V\ni\eta\). Then \(U^{-1}V\) is a bisection containing \(\gamma^{-1}\eta\) and \(c(U^{-1}V) = 0+\lambda = \lambda\). Let \(b\in \A_{\gamma}\). Since the Fell bundle \(\A\) is saturated,
\[
 \A_{\eta}\A^*_{\gamma^{-1}\eta} =\textup{span}\{ac^*: a\in \A_{\eta} \textup{ and } c\in \A_{\gamma^{-1}\eta} \}
\] 
 is dense in \(\A_{\gamma}\). Thus, there exists \(a_1,a_2,\cdots,a_n\in \A_{\gamma}\) and \(c_1,c_2,\cdots,c_n\in \A_{\gamma^{-1}\eta}\) such that 
\begin{equation}\label{eq-crys-dense}
\Big\|b- \sum_{i=1}^{n}a_ic_i^*\Big\|<\epsilon
\end{equation}
for any given \(\epsilon >0\). Again, the Fell bundle \(\A\) has enough section, choose continuous sections \(t_i\in \Cc(V;\A|_{V})\) and \(k_i\in \Cc(U^{-1}V;\A|_{U^{-1}V})\) such that \(t_i(\eta) = a_i\) and \(k_i(\gamma^{-1}\eta) = c_i\) for \(i=1,2,\cdots,n\). Thus, Equation~\eqref{eq-crys-dense} gives us
\[
 \Big\|b- \sum_{i=1}^{n}t_i*k_i^*(\gamma)\Big\| <\epsilon.
\]
Therefore, by~\cite[Lemma A. 4]{Muhly-Williams2008Equivalence-and-disinte-thm-Fell-bundle}, we have \(\overline{J} = \Contz(Y;\A|_{Y})\). Combining with \(I=\overline{J}\), we have \(I= \overline{\Cc(Y;\A|_{Y})} =\ker(\rho) \). Consequently,
\[
A_c =\frac{A_0}{I} =\frac{\Cst(K;\A|_K)}{\ker(\rho)}
\cong \Cst(G(Z);\A|_{G(Z)}).
\]
Finally, if \(Z\neq\emptyset\), then
\(\Cst(G(Z);\A|_{G(Z)})\neq0\), since the unit fibres
\(\A_x\) are nonzero for \(x\in Z\). Hence \(A_c\neq0\).
Conversely, if \(Z=\emptyset\), then \(G(Z)=\emptyset\) and
\(\Cst(G(Z);\A|_{G(Z)})=0\), so \(A_c=0\).
 \end{proof}

\section{Applications and examples}\label{sec-application}
In this section, we apply the main result, Theorem~\ref{thm-main-ground-st}, to characterize ground states on several classes of \(\Cst\)-algebras, including twisted \(\Cst\)-algebras of Deaconu--Renault groupoids, twisted higher-rank graph \(\Cst\)-algebras, and groupoid crossed products.

\subsection{Ground states on twisted groupoid \(\Cst\)-algebras}
Let \(G\) be a locally compact Hausdorff \'etale groupoid and let
\(\tau\in Z^2(G;\mathbb T)\) be a \(2\)-cocycle. Consider \(
\A=G\times \C\)
equipped with the product topology. Then \(\A\) is a Fell bundle over
\(G\), with bundle map \(
p\colon\A\to G\) given by \(p(\gamma,z) =\gamma\)
and operations are given by 
\[
(\gamma, z)(\eta, w) = (\gamma\eta, \tau(\gamma, \eta)zw) \quad \textup{and} \quad (\gamma, z)^* = (\gamma^{-1}, \overline{\tau(\gamma, \gamma^{-1})}\overline{z}).
\] 
This is called the \emph{Fell line bundle} associated to
\(\tau\). The Fell bundle \(\Cst\)\nb-algebra
\(\Cst(G;\A)\) is canonically isomorphic to the twisted groupoid
\(\Cst\)\nb-algebra \(\Cst(G;\tau)\) (see~\cite[Lemma 4.1]{Afsar-Sims2021KMS-state-on-Cst-alg-Fell-bundle-over-gpd}).

Let \(c\colon G\to\R\) be a \(1\)-cocycle, and let
\(\sigma^c\) be the associated real dynamics on \(\Cst(G;\tau)\). Let
\(Z\subseteq \base\) be the boundary set associated to \(c\). The
restriction of \(\tau\) to \(G(Z)^{(2)}\) is again a \(2\)-cocycle, which we
continue to denote by \(\tau\). Moreover, the restricted Fell bundle
\(\A|_{G(Z)}\) is the Fell line bundle associated to this
restricted cocycle. Consequently,
\[
\Cst(G(Z);\A|_{G(Z)})
\cong \Cst(G(Z);\tau).
\]
Thus, as a consequence of Theorem~\ref{thm-main-ground-st}, we obtain the
following.

\begin{corollary}\label{coro-ground-gpd-alg}
	Let \(G\) be a locally compact Hausdorff second countable \'etale groupoid and
	let \(\tau\in Z^2(G;\mathbb T)\) be a \(2\)-cocycle.  Let
	\(c\colon G\to\mathbb R\) be a \(1\)-cocycle  with the associated real dynamics \(\sigma^c\) given by \( \sigma^c_t(f)(\gamma)
	=
	e^{itc(\gamma)}f(\gamma)\) for
	\(f\in \Cc(G; \tau)\) and \(\gamma\in G\).
	If \(Z\neq \emptyset\) and \(G(Z)\) is \'etale, then there is an affine homeomorphism between the state space of \(\Cst(G(Z);\tau)\) and the
	\(\sigma^c\)-ground state space of \(\Cst(G;\tau)\).
	
	More precisely, this affine homeomorphism sends a state
	\(\varphi\) on \(\Cst(G(Z);\tau)\) to the unique
	\(\sigma^c\)-ground state \(\psi_\varphi\) on \(\Cst(G;\tau)\) satisfying
	\[
	\psi_\varphi(f)
	=
	\varphi(f|_{G(Z)}),
	\quad f\in \Cc(G; \tau).
	\]
	If \(Z=\emptyset\), then \(C^*(G;\tau)\) has no
	\(\sigma\)-ground states.
\end{corollary}
\begin{proof}
	Consider the Fell line bundle \(\A=G\times\C\) associated to
	the \(2\)-cocycle \(\tau\). Its restriction to \(G(Z)\) is the Fell line
	bundle associated to \(\tau|_{G(Z)^{(2)}}\), and we have
	\[
	\Cst(G(Z);\A|_{G(Z)})
	\cong \Cst(G(Z);\tau).
	\]
	Thus, the result follows immediately from
	Theorem~\ref{thm-main-ground-st}.
	Under this identification, the state \(\varphi\) on \(\Cst(G(Z); \tau)\) is extended to a \(\sigma^c\)-ground state
	\[
	\psi_\varphi(f)=\varphi(f|_{G(Z)})
	\]
	for \(f\in \Cc(G; \tau)\).
	If \(Z=\emptyset\), the final assertion follows from the corresponding
	statement in Theorem~\ref{thm-main-ground-st}.
\end{proof}

\subsection{Ground states on twisted \(\Cst\)-algebras of Deaconu--Renault groupoids}
Let \(X\) be a locally compact Hausdorff space and \(T\colon X\to X\) a local homeomorphism. Then
\[
G_T= \{(x,m-n,y) \in X\times \Z\times X: m,n\geq 0 \textup{ and } T^m(x) =T^n(y)\}
\]
has an \'etale groupoid structure given by
\begin{itemize}
	\item \((x,m-n,y)\) and \((z,p-q,w)\) are composable if and only if \(y=z\) and the composiation is given by  \((x,m-n,y)(y, p-q,w) = (x, m+p-(n+q), w)\);
	\item \((x,m-n,y)^{-1} = (y,n-m,x)\).
\end{itemize}
The range and source maps are given by \(r(x,m-n,y) = (x,0,x)\) and \(s(x,m-n,y) = (y,0,y)\). The unit space of \(G_T\) can be identified with \(X\). This groupoid is called \emph{Deaconu--Renault groupoid} (see~\cite{Deaconu-1995-Groupoid-associated-endomor}). Note that if \(X\) is second countable then so is \(G_T\). 
For a minimal Deaconu--Renault system, Proposition~4.2 of~\cite{Armstong-Sims-Simplicity-twist-Deaconu-Renault-gpd} produces a collection of continuous \(2\)-cocycle \(Z^2(G_T, \mathbb{T})\) from the bicharacter of \(P_T\) where 
\[
P_T= \{p\in \Z : (x, p, x) \in G_T \textup{ for all } x\in X\}
\] 
is a subgroup of \(\Z\).
 
Fix \(\omega \in Z^2(G_T, \mathbb{T})\). Consider the the canonical \emph{gauge cocycle} \(c\colon G_T\to \R\) by \(c(x,m-n,y) = m-n\). This cocycle induces a dynamics \(\sigma^c\) on \(\Cst(G_T; \omega)\) by 
\[
\sigma^c_t(f)(x,m-n,y) = \e^{it(m-n)} f(x,m-n,y)
\]
for \(f\in \Cc(G_T; \omega)\). The boundary set \(Z\) for \(c\) is 
\[
Z=\{y\in X :c(\gamma) \geq 0 \textup{ for } \gamma \in (G_T)_y\}.
\]
Let \(y\in X\) and \(n\in \N\). Consider an arrow \(\gamma = (T^n(y), -n, y) \in G_T\) and \(s(\gamma) = y\). Then \(c(\gamma) = -n <0\). Thus, \(y\notin Z\), consequently \(Z=\emptyset\).

\noindent Therefore, Corollary~\ref{coro-ground-gpd-alg} ensures that for the canonical gauge cocycle \(c\), the twisted \(\Cst\)\nb-algebra \(\Cst(G_T; \omega)\) has no \(\sigma^c\)\nb-ground states.

\subsubsection{A more interesting cocycle} 
To obtain a nonempty boundary set, we consider cocycle from a \emph{potential} \(F\). Let
\(F\colon X\to\R\) be a continuous function, and for \(n\in\N\) define
\[
F_0(x)=0 \quad \textup{and} \quad 
F_n(x)=\sum_{j=0}^{n-1}F(T^j(x)), \quad \textup{for } n\geq 1.
\]
Then \(
c_F(x,m-n,y)=F_m(x)-F_n(y) \)
is a continuous \(1\)-cocycle on \(G_T\), and hence determines a
one-parameter automorphism group \(\sigma^F\) on \(\Cst(G_T; \omega)\) by
\begin{equation}\label{eq-dyna-potential}
\sigma^F_t(f)(x,m-n,y)
=
e^{it(F_m(x)-F_n(y))}f(x,m-n,y),
\end{equation}
for \( f\in \Cc(G_T; \omega).\) 
Our goal is to understand the boundary set 
\[
Z_F= \{y\in X : c_F(\gamma) \geq 0 \textup{ for } \gamma \in (G_T)_y\}.
\]
Note that for \(y\in X\), an element of the source fibre \((G_T)_y\) is of the form \((x, m-n,y)\) with \(T^m(x) = T^n(y)\). Thus,
\begin{equation}\label{eq-DR-boun-1}
	y\in Z_F \iff F_n(y)\leq F_m(x) \textup{ whenever } T^n(y) = T_m(x).
\end{equation}
For \(z\in X\), consider the set \(H_z=\{(x,n) \in X\times \N : T^n(x) = z \}\). Define 
\begin{align*}
	M(z) &= \inf \{ F_n(x) : (n,x) \in H_z\}\\
	      &=\inf \{F_n(x) : n\in \N \textup{ and } T^n(x) = z\}.
\end{align*}
Note that \((z,0) \in H_z\) and \(F_0(z) =0\), thus \(M(z) \leq 0\).

Let \(y\in Z_F\) and fix \(n\in \N\). Set \(z= T^n(y)\). Then the pair \((y,n) \in H_z\). Since \(y\in Z_F\), Equation~\eqref{eq-DR-boun-1} gives us \(F_n(y)\leq F_m(x)\) for \(T^n(y) = T^m(x)\). Thus, we have 
\[
 F_n(y) \leq F_m(x) \quad \textup{for every pair } (x,m) \textup{ with } z= T^m(x).
\] 
Now taking infimum over all such \((x,m)\), we obtain \(	F_n(y)\leq M(z).\)
Again by definition we have \(M(z) \leq F_n(y)\), consequently we have 
\begin{equation}\label{eq-DR-boun-2}
	 F_n(y) = M(z) =M(T^n(y)) \quad \textup{ for } n\in \N.
\end{equation}

Conversely, suppose that Equation~\eqref{eq-DR-boun-2} holds for a fix \(y\in X\). Let \(\gamma_y = (x, m-n,y) \in (G_T)_y\). Then by definition of \(M(z)\), we have \(M(z)\leq F_m(x)\) and by the assumption we have \(F_n(y)\leq F_m(x)\). Therefore, Equation~\eqref{eq-DR-boun-1} gives us \(y\in Z_F\). Therefore, we have proved the following 
\begin{equation}\label{eq-DR-boun-3}
	y\in Z_F \iff F_n(y) = M(T^n(y)) \quad \textup{ for all } n\in \N.
\end{equation}
Hence, the boundary groupoid \(G_T(Z_F)\) can be identified with
\begin{equation}\label{eq-DR-boun-4}
	G_T(Z_F) = \{(x,m-n,y) : x,y \in Z_F, T^m(x) = T^n(y) \textup{ and } F_n(y) = F_m(x)\}.
\end{equation}
Thus, we have proved the following result.
\begin{corollary}\label{coro-ground-st-DR-gpd}
		Let \(X\) be a locally compact Hausdorff second countable space and
	let \(T\colon X\to X\) be a local homeomorphism. Let \(G_T\) be the associated Deaconu--Renault groupoid and \(\omega \in Z^2(G_T; \mathbb{T})\). Suppose \(F\colon G_T\to \R\) is a continuous function with the associated \(1\)\nb-cocycle \(c_F(x,m-n,y)=F_m(x)-F_n(y) \). Then the boundary set \(Z_F\) is characterize by Equation~\eqref{eq-DR-boun-3}. 

	 If \(Z_F\neq\emptyset\), there is an affine homeomorphism between the \(\sigma_F\)\nb-ground state space of \(\Cst(G_T; \omega)\) and state space of \(\Cst(G_T(Z_F); \omega)\), where \(G_T(Z_F)\) is the boundary groupoid given by Equation~\eqref{eq-DR-boun-4}.
\end{corollary}

\begin{example}[Strictly positive potential]
Let \(X\) be a locally compact Hausdorff second countable space and \(F\colon X\to \R\) a continuous map with \(F(x)>0\) for all \(x\in X\). For \(y\in X\) and \(n\in \N\), set \(\gamma_y(T^n(y),-n,y) \in (G_T)_y\). Then \(c_F(\gamma_y) = -F_n(y) <0\) for all \(n\in \N\). Thus, \(y\notin Z_F\) and hence \(Z_F=\emptyset\). Therefore, there does not exist any \(\sigma^F\)\nb-ground states on \(\Cst(G_T; \omega)\). 
\end{example}

\subsection{Ground states on twisted \(\Cst\)-algebras of higher-rank graphs}
Higher-rank graphs (or, \(k\)-graphs) were introduced by Kumjian and Pask
in~\cite{Kumjian-Pask2000Higher-Rank-graph-cst-alg} as a
higher-dimensional generalization of Cuntz--Krieger algebras. A
\emph{\(k\)-graph} is a countable category \(\Lambda\) together with a
functor \(
d\colon \Lambda\to\N^k \)
satisfying the factorization property: for every
\(\lambda\in\Lambda\) with \(d(\lambda)=m+n\), there exist unique
\(\mu,\gamma\in\Lambda\) such that \(
s(\mu)=r(\gamma), d(\mu)=m, d(\gamma)=n\) and 
\(\lambda=\mu\gamma.\)
The vertex set of the \(k\)-graph \(\Lambda\), denoted by \(\Lambda^0\), can be
identified with \(d^{-1}(0)\). For \(p\in \N^k\), define \(\Lambda^p \defeq \{\lambda \in \Lambda : d(\lambda) = p\}\) and \(\Lambda\) is said to be finite if \(\Lambda^p\) is finite for all \(p\in \N^k\). Given vertices \(u, v\in \Lambda^0\), we write \(v\Lambda u\defeq \{\lambda \in \Lambda : r(\lambda) = v  \textup{ and } s(\lambda) = u\}\). The \(k\)\nb-graph \(\Lambda\) is called \emph{strongly connected} if \(v\Lambda u \neq \emptyset\) for all vertices \(u,v\in \Lambda^0\). A \(k\)\nb-graphs \(\Lambda\) has \emph{no sources} if \(v\Lambda^p \neq \emptyset\) for all \(p\in \N^k\) and \(\Lambda\) is called \emph{row-finite} if \(v\Lambda^p\) is finite for all \(p\in \N^k\). 

Let \(
\Lambda^{(2)}
=
\{(\lambda,\mu)\in\Lambda\times\Lambda:s(\lambda)=r(\mu)\}.\)
A map \(\omega\colon\Lambda^{(2)}\to\mathbb{T}\) is called a
\(\mathbb T\)-valued \(2\)-cocycle on \(\Lambda\) if
\[
\omega(r(\lambda),\lambda)
=\omega(\lambda,s(\lambda)) =1 \quad \textup{and} \quad 
\omega(\lambda,\mu)\omega(\lambda\mu,\nu) =
\omega(\lambda,\mu\nu)\omega(\mu,\nu)
\]
for all composable triples \((\lambda,\mu,\nu)\). We denote the set of
all such \(2\)-cocycles by \(Z^2(\Lambda,\mathbb{T})\). For a row-finite
\(k\)-graph \(\Lambda\) and \(\omega\in Z^2(\Lambda,\mathbb{T})\), the
associated \emph{twisted \(k\)-graph \(\Cst\)-algebra} is denoted by
\(\Cst(\Lambda,\omega)\); we refer the reader to~\cite[Section~5]{Kumjian-Pask-Sims2015Twisted-Higher-rank-graph-alg}
for its definition.

Let \(
\Omega_k = \{(m,n)\in\N^k\times\mathbb N^k:m\leq n\}\).
Then \(\Omega_k\) is a \(k\)\nb-graph with structure maps \(r(p,q) = (p,p)\), \(s(p,q) = (q,q)\), \((p,q)(m,n) = (p,n)\) and \(d(p,q) = p-q\).
The vertex set \(\Omega_k^0\) can be identified with \(\N^k\).
The \emph{infinite-path space} of a \(k\)-graph \(\Lambda\) is
\[
\Lambda^\infty
= \{x\colon\Omega_k\to\Lambda:
x\textup{ is a degree maps preserving functor}\}.
\]
For \(l\in\mathbb N^k\), the shift map \( 
\kappa^l\colon\Lambda^\infty\to\Lambda^\infty \)
is defined by \(\kappa^l(x)(p,q)  = x(p+l, q+l)\).

Suppose that \(\Lambda\) is row-finite \(k\)-graph and has no sources. The
\emph{infinite-path groupoid} of \(\Lambda\) is
\[
G_{\Lambda}= \big\{(x,l,y) \in  \Lambda^\infty\times\Z^k\times\Lambda^\infty:  l= p-q \textup{ for some } p,q\in \N^{k} \textup{ and } \kappa^p(x) =\kappa^q(y)\big\}
\]
The groupoid operations are given by
\[
r(x,l,y)=x,~s(x,l,y)=y,~ 
(x,l,y)(y,l',z)
=(x,l+l',z), \textup{ and }
(x,l,y)^{-1}=(y,-l,x).
\]
Thus, the unit space \(G_\Lambda^{(0)}\) of the path groupoid  can be identified with \(\Lambda^\infty\).

For \(\lambda,\mu\in\Lambda\) satisfying
\(s(\lambda)=s(\mu)\), define
\[
Z(\lambda,\mu)=
\left\{
(\lambda x,d(\lambda)-d(\mu),\mu x) \in G_{\Lambda}:
x\in\Lambda^\infty,\ r(x)=s(\lambda)
\right\}.
\]
The collection \(\{Z(\lambda,\mu): \lambda, \mu \in \Lambda \textup{ and } s(\lambda)=s(\mu)\}\)
forms a basis for a topology on \(G_{\Lambda}\).
Moreover, \(G_\Lambda\) is an \'etale groupoid; see
\cite[Proposition~2.8]{Kumjian-Pask2000Higher-Rank-graph-cst-alg}.

Let \(
\Lambda{_s *_s} \Lambda := \{(\lambda,\mu) \in \Lambda \times \Lambda : s(\lambda)=s(\mu)\}\) and let \(\mathcal{P} \subseteq \Lambda {_s *_s} \Lambda\) be such that \((\lambda,s(\lambda)) \in \mathcal{P}\) for every \(\lambda \in \Lambda\) and
\[
G_{\Lambda} = \bigsqcup_{(\lambda,\mu)\in \mathcal{P}} Z(\lambda,\mu).
\]
We can always choose such \(\mathcal{P}\) by \cite[Lemma~6.6]{Kumjian-Pask-Sims2015Twisted-Higher-rank-graph-alg}. Fix \(\alpha \in G_{\Lambda}\), let \((\lambda_\alpha,\mu_\alpha)\) denote the unique element of \(\mathcal{P}\) with \(\alpha \in Z(\lambda_\alpha,\mu_\alpha)\), and define \(f\colon G_{\Lambda} \to \mathbb{Z}^k\) by \(f(x,n,y)=n\). 
Let \(\omega\in Z^2(\Lambda,\mathbb T)\) be a \(2\)\nb-cocycle. By
\cite[Lemma~6.3]{Kumjian-Pask-Sims2015Twisted-Higher-rank-graph-alg},
for every composable pair
\((\alpha,\beta)\in G_\Lambda^{(2)}\), there exist
\(\nu,\xi,\zeta\in\Lambda\) and \(y\in\Lambda^\infty\) such that
\[
\mu_\alpha\nu = \lambda_\beta\xi, \quad
\lambda_\alpha\nu = \lambda_{\alpha\beta}\zeta, \quad
\mu_\beta\xi = \mu_{\alpha\beta}\zeta,
\]
and
\begin{align*}
	\alpha &= (\lambda_\alpha\nu y,f(\alpha),\mu_\alpha\nu y),\\
	\beta  &= (\lambda_\beta\xi y,f(\beta),\mu_\beta\xi y),\\
	\alpha\beta &= (\lambda_{\alpha\beta}\zeta y, f(\alpha\beta),
	\mu_{\alpha\beta}\zeta y).
\end{align*}
Using these data, we define \(\tau_\omega\colon G_\Lambda^{(2)}\to\mathbb{T}\)
by
\begin{equation}\label{eq-2-cocycle-gpd}
\tau_\omega(\alpha,\beta)
		=
		\omega(\lambda_\alpha,\nu)
		\overline{\omega(\mu_\alpha,\nu)}
		\omega(\lambda_\beta,\xi)
		\overline{\omega(\mu_\beta,\xi)}
	\overline{\omega(\lambda_{\alpha\beta},\zeta)}
		\omega(\mu_{\alpha\beta},\zeta).
\end{equation}
By~\cite[Lemma~6.3]{Kumjian-Pask-Sims2015Twisted-Higher-rank-graph-alg}, \(\tau_\omega\) defines a continuous \(2\)-cocycle on \(G_\Lambda\), which independent of
the choices of \(\nu,\xi,\zeta\). Moreover,
\cite[Theorem~6.5]{Kumjian-Pask-Sims2015Twisted-Higher-rank-graph-alg}
shows that the cocycles obtained from different choices of
\(\mathcal{P}\) are cohomologous. Finally, by
\cite[Corollary~7.8]{Kumjian-Pask-Sims2015Twisted-Higher-rank-graph-alg},
there exists an isomorphism
\[
\Cst(\Lambda,\omega) \cong
\Cst(G_\Lambda,\tau_\omega).
\]
For a \(k\)-graph \(\Lambda\) and \(1\leq i\leq k \), let \(A_i\in \Mat_{\Lambda^0}\) be the matrix with entries \(A_i(v,u) = |v\Lambda^{e_i}u|\). Let \(\rho(A_i)\) be the spectral radius of \(A_i\). Define a map \(c\colon G_{\Lambda} \to \R\) by 
\begin{equation}\label{eq-cocycle-graph}
c(x,n,y) = \sum_{i=1}^{k} n_i\ln(\rho(A_i)).
\end{equation}
The map \(c\) is locally constant and thus \(c\) is a continuous \(1\)\nb-cocyle. For \(\tau\in Z^2(G_\Lambda, \mathbb{T})\), the cocycle \(c\) induces a dynamics \(\sigma^c\) on \(\Cst(G_{\Lambda}; \tau)\) given by 
\[
 \sigma^c_t(f)(x,n,y) = \e^{itc(x,n,y)} f(x,n,y)
\] 
for \(f\in \Cc(G_{\Lambda}; \tau)\). The dynamics is referred to as the \emph{preferred dynamics} in~\cite{an-Huef-Laca-Raeburn-Sims-KMS-states-k-graph-periodicity}.

\begin{lemma}\label{lem-char-of-bound-set-k-graph}
	Let \(\Lambda\) be a strongly connected row-finite \(k\)-graph with no sources, and let \(c\) be the \(1\)-cocycle on \(G_{\Lambda}\) given by Equation~\eqref{eq-cocycle-graph}. Then the boundary set \(Z=\Lambda^{\infty}\) if and only if \(\rho(A_i) = 1\) for all \(i=1,2,\cdots,k\).
\end{lemma}
\begin{proof}
	Suppose that \(\rho(A_i) = 1\) for all \(i=1,2,\cdots,k\).
	For \(g=(y,m-n,x)\in G_\Lambda\), Equation~\eqref{eq-cocycle-graph}
	gives
	\[
	c(g)=\sum_{i=1}^k(m_i-n_i)\ln(\rho(A_i))=0.
	\]
	Thus, \(c(g)\geq0\) for every \(g\in G_\Lambda\). Hence, every unit
	\(x\in Z\). Therefore, \(Z=\Lambda^{\infty}\).
	
	Conversely, suppose that \(Z=\Lambda^{\infty}\).
	Let \(x\in\Lambda^\infty\) and \(p\in\N^k\). Then \(g_p=(\kappa^p(x),-p,x)\in G_\Lambda\) and 
	since \(s(g_p)=x\in Z\), we have
	\[
   c(g_p) = -\sum_{i=1}^k p_i\ln(\rho(A_i))\geq0.
	\]
	Thus, we have
	\begin{equation}\label{eq-boundary-inequality-1}
		\sum_{i=1}^k p_i\ln(\rho(A_i))\leq0.
	\end{equation}
	On the other hand, since \(\Lambda\) has no sources, for every \(p\in\N^k\) we may choose
	\(\lambda\in\Lambda^p\) such that \(s(\lambda)=r(x)\). Then
	\(\lambda x\in\Lambda^\infty\) and \(\kappa^p(\lambda x)=x.\)
	Consequently, \(g^p=(\lambda x,p,x)\in G_\Lambda.\)
	Again \(s(g^p)=x\in Z\), so \(c(g^p)\geq0\). Thus, we have
	\begin{equation}\label{eq-boundary-inequality-2}
		\sum_{i=1}^k p_i\ln(\rho(A_i))\geq0.
	\end{equation}
	Combining Equations~\eqref{eq-boundary-inequality-1} and~\eqref{eq-boundary-inequality-2}, we obtain
	\[
	\sum_{i=1}^k p_i\ln(\rho(A_i)) = 0
	\]
	for all \(p\in \N^k\).
	Taking \(p=e_i\), the \(i\)-th standard basis vector of
	\(\N^k\), gives \(\ln(\rho(A_i)) =0 \) and hence \(\rho(A_i) =1\) for all \(i=1,2,\cdots,k\).
\end{proof}

\noindent The next theorem gives a complete characterization of the ground state for twisted \(k\)-graph algebras for preferred dynamics. 

\begin{theorem}\label{thm-ground-k-graph}
	Let \(\Lambda\) be a strongly connected row-finite \(k\)-graph with no source and \(\omega\in Z^2(\Lambda, \mathbb{T})\). Let \(c\) be a \(1\)\nb-cocycle on \(G_{\Lambda}\) given by Equation~\eqref{eq-cocycle-graph} and let \(\sigma^c\) be the preferred dynamics on \(\Cst(G_{\Lambda}; \tau_\omega)\).
	Then the following holds.
	\begin{enumerate}
		\item  If \(\rho(A_i) = 1\) for \(i=1,2,\cdots,k\), then any state on \(\Cst(\Lambda; \omega)\) is a \(\sigma^c\)\nb-ground state on \(\Cst(\Lambda; \omega)\).
		\item If \(\rho(A_i)\neq 1\) for at least one \(i\), then \(\Cst(\Lambda; \omega)\) has no \(\sigma^c\)\nb-ground states.
	\end{enumerate}
\end{theorem}
\begin{proof}
	Since \(\Cst(\Lambda;\omega) \cong \Cst(G_\Lambda;\tau_\omega)\),
we can apply Corollary~\ref{coro-ground-gpd-alg} to describe the ground states.

\noindent (1). Suppose that \(\rho(A_i)=1\) for every \(i =1,2,\cdots,k\). Then by Lemma~\ref{lem-char-of-bound-set-k-graph}
\(Z=\Lambda^\infty\), and hence the boundary groupoid \(
G_\Lambda(Z)=G_\Lambda\).
Thus, Corollary~\ref{coro-ground-gpd-alg} identifies the
\(\sigma^c\)-ground state space with the state space of
\(\Cst(G_\Lambda;\tau_\omega)\). Consequently, every state on
\(\Cst(\Lambda;\omega)\) is a \(\sigma^c\)-ground state.

\noindent (2). Suppose \(\rho(A_j) \neq 1\) for some \(j\). We complete the proof in two separate cases.

\textbf{Case I:} Assume \(\rho(A_j)>1\). 
Let \(x\in \Lambda^{\infty}\). The shift \(\kappa^{e_j}(x)\) is well-defined as the \(k\)\nb-graph has no source. And we have \(g=(\kappa^{e_j}(x), -e_j, x)\in G_{\Lambda}\) and \(s(g) =x\), where \(e_j\) is standard basis vector for \(\N^k\). Then \(c(g) = -\ln(\rho(A_j)) <0\). Thus, \(x\notin Z\). Since \(x\in\Lambda^\infty\) was arbitrary, \(Z=\emptyset\).
Therefore, the boundary groupoid \(G_{\Lambda}(Z) = \emptyset\). Hence, by Corollary~\ref{coro-ground-gpd-alg}, \(\Cst(\Lambda; \omega)\) has no \(\sigma^c\)-ground states.

\textbf{Case II:} Assume \(\rho(A_j)<1\). Let \(x\in \Lambda^{\infty}\). Since \(\Lambda\) has no source, choose $\lambda \in \Lambda^{e_j}$ with \(s(\lambda) =r(x)\). Then \(g_x=(\lambda x, e_j, x) \in G_{\Lambda}\) as \(\kappa^{e_j}(\lambda x) =x\). Now \(s(g_x) =x\) and \(c(g_x) = \ln(\rho(A_j)) <0\). Therefore, \(x\notin Z\). As, \(x\in \Lambda^{\infty}\) was arbitrary, \(Z=\emptyset\). Then by Corollary~\ref{coro-ground-gpd-alg}, \(\Cst(\Lambda; \omega)\) has no \(\sigma^c\)-ground states.
\end{proof}

\subsection{Ground states on groupoid crossed products}
In this subsection, we will apply Theorem~\ref{thm-main-ground-st} to the Fell bundle associated with a groupoid crossed product.
\begin{definition}[{\cite[Defintion 3.46]{Goehle2009Phd-Thesis-Gpd-cros-Prod}}]
A groupoid dynamical system is a triple \((A,G, \alpha)\), where \(A\) is a \(\Contz(\base)\)\nb-algebra and \(\alpha = \{\alpha_{\gamma}\}_{\gamma\in G}\) is an action of \(G\) on \(A\) satisfying
\begin{enumerate}
	\item for each \(\gamma \in G\), \(\alpha_{\gamma}\colon A(s(\gamma)) \to A(r(\gamma))\) is a \(^*\)-isomorphism;
	\item \(\alpha_{\gamma\eta} = \alpha_{\gamma}\circ\alpha_{\eta}\);
	\item  the map from \(G\times_{s, \base,p}\A\to \A\) given by \((\gamma, a) \mapsto \alpha_{\gamma}(a)\) is continuous, where \(p\colon\A\to \base\) is the upper semicontinuous bundle of \(\Cst\)\nb-algebras associated to the \(\Contz(\base)\)\nb-algebra \(A\).  
\end{enumerate}
\end{definition}
We fix a groupoid dynamical system \((A,G, \alpha)\), where \(G\) is an \'etale groupoid. Let \(q\colon r^*\A \to G\) be the \emph{pull back} bundle along with the range map, i.e.,
\[
r^*\A = \{(\gamma,a) \in G\times \A : r(\gamma) = p(a)\}
\] 
and \(q(\gamma,a) = \gamma\). Then~\cite[Example 2.1]{Muhly-Williams2008Equivalence-and-disinte-thm-Fell-bundle} ensures that \(r^*\A\) is a Fell bundle over \(G\) and the groupoid crossed product \(A\rtimes_{\alpha}G\) is isomorphic to the Fell bundle \(\Cst\)\nb-algebra \(\Cst(G;r^*\A)\) (by~\cite[Example 2.8]{Muhly-Williams2008Equivalence-and-disinte-thm-Fell-bundle}).

Let \(c\colon G\to \R\) be a \(1\)\nb-cocycle and let \(Z\) be the associated boundary set. 
For \(\gamma\in G(Z)\), the fibre of the restricted Fell bundle is \((r^*\A|_{G(Z)})_\gamma=\A_{r(\gamma)}\).
Therefore, \(r^*\A|_{G(Z)}\) is precisely the Fell bundle associated to the
restricted groupoid dynamical system \((A|_Z,G(Z),\alpha|_{G(Z)} ).\)
Thus, on the dense subalgebra of compactly supported sections, we
have the canonical identification
\[
\Cc(G(Z);r^*\A|_{G(Z)})
\cong
\Cc(G(Z),\A|_Z).
\]
 Hence, by the
disintegration theorem for Fell bundles
\cite[Theorem 4.13]{Muhly-Williams2008Equivalence-and-disinte-thm-Fell-bundle}
and the disintegration theorem for groupoid crossed products
\cite[Theorem 7.12]{Muhly-Williams2008Renaults-equivalence-thm-for-gpd-cros-prd},
this identification extends to a canonical \(\Cst\)\nb-algebra isomorphism
\[
\Cst(G(Z);r^*\A|_{G(Z)})
\cong
A|_Z\rtimes_{\alpha|_{G(Z)}}G(Z).
\]
\begin{corollary}[Corollary of Theorem~\ref{thm-main-ground-st}]
	Let \((A,G, \alpha)\) be a groupoid dynamical system where \(G\) is a locally compact Hausdorff second countable \'etale groupoid and let \(c\colon G\to \R\) be a \(1\)\nb-cocycle. Suppose \(\sigma^c\) is the real dynamics on \(A\rtimes_{\alpha}G\) associated with~\(c\) as in Equation~\eqref{equ:real-dyna}. If \(Z\neq \emptyset\) and the boundary groupoid \(G(Z)\) is \'etale, then there is an affine homeomorphism between the state space of \(A|_Z\rtimes_{\alpha|_{G(Z)}}G(Z)\) and the \(\sigma^c\)\nb-ground state space of \(A\rtimes_{\alpha}G\). This assignment sends a state \(\varphi\) on \(A|_Z\rtimes_{\alpha|_{G(Z)}}G(Z)\) to a \(\sigma^c\)\nb-ground state \(\psi_{\varphi}\) on \(A\rtimes_{\alpha}G\) satisfying Equation~\eqref{eq-relation-state-ground}. If \(Z=\emptyset\), then \(A\rtimes_{\alpha}G\) has no \(\sigma^c\)\nb-ground states.
\end{corollary}

\medskip
\paragraph{\itshape Acknowledgements:}
We are grateful to S. Sundar for many fruitful discussions, which led to this article.



\end{document}